\documentclass[a4paper,12pt]{amsart}
\usepackage{amsmath,latexsym,amssymb,amsfonts}
\usepackage{amscd,amssymb,epsfig}
\usepackage[arrow, matrix, curve]{xy}
\usepackage{hyperref}
\usepackage{mathdots}

\numberwithin{equation}{section}

\newtheorem{theorem}{Theorem}[section]
\newtheorem{lemma}[theorem]{Lemma}
\newtheorem{corollary}[theorem]{Corollary}
\newtheorem{proposition}[theorem]{Proposition}

\theoremstyle{definition}
\newtheorem{definition}[theorem]{Definition}

\newtheorem{remark}[theorem]{Remark}

\renewcommand{\epsilon}{\varepsilon}

\newcommand{\A}{\mathcal{A}}
\newcommand{\B}{\mathcal{B}}
\newcommand{\C}{\mathbb{C}}
\newcommand{\D}{\mathbb{D}}
\newcommand{\E}{\mathcal{E}}
\newcommand{\F}{\mathcal{F}}
\renewcommand{\H}{\mathcal{H}}

\newcommand{\M}{\mathcal{M}}
\newcommand{\N}{\mathbb{N}}

\renewcommand{\P}{\mathbb{P}}
\newcommand{\R}{\mathbb{R}}
\renewcommand{\S}{\mathcal{S}}

\newcommand{\id}{\operatorname{id}}
\newcommand{\op}{{\operatorname{op}}}

\newcommand{\vN}{\operatorname{vN}}
\newcommand{\alg}{{\operatorname{alg}}}
\newcommand{\fin}{{\operatorname{fin}}}

\makeindex

\title[Regularity of distributions of Wigner integrals]{Regularity of distributions of Wigner integrals}

\author[T. Mai]{Tobias Mai}
\address{Universit\"{a}t des Saarlandes, FR $6.1-$Mathematik, 66123 Saarbr\"{u}cken, Germany }
\email{mai@math.uni-sb.de}

\date{\today}

\thanks{This work was supported by the ERC Advanced Grant "Non-commutative distributions in free probability" and by a grant from the DFG (SP-419-8/1).\\
The author wants to express his thanks to Roland Speicher for many inspiring discussions on the topic and his valuable comments during the preparation of this article. Moreover, the author is grateful to Yoann Dabrowski for some useful remarks at the final stage, to Mehmet Madensoy for bringing some of the references that are listed at the end to the author's attention, and to S\"oren M\"oller, together with whom the author once started to learn free stochastic calculus. Finally, the author wants to thank Ivan Nourdin and Giovanni Peccati for very interesting and fruitful discussions on this article and beyond.}

\keywords{free probability theory, Wigner integrals, free Malliavin calculus, non-commutative derivations, zero-divisors, absence of atoms}

\subjclass[2000]{46L54 (46L53, 46L57)}

\begin{document}

\begin{abstract}
Wigner integrals and the corresponding Wigner chaos were introduced by P. Biane and R. Speicher in 1998 as a non-commutative counterpart of classical Wiener-It\^o integrals and the corresponding Wiener-It\^o chaos, respectively, in free probability.

In the classical case, a famous result of I. Shigekawa states that non-trivial elements in the finite Wiener-It\^o chaos have an absolutely continuous distribution. We provide here a first contribution to such regularity questions for Wigner integrals by showing that the distribution of non-trivial elements in the finite Wigner chaos cannot have atoms. This answers a question of I. Nourdin and G. Peccati.

For doing so, we establish the notion of directional gradients in the context of the free Malliavin calculus. These directional gradients bridge between free Malliavin calculus and the theory of non-commutative derivations as initiated by D. Voiculescu and Y. Dabrowski. Methods recently invented by R. Speicher, M. Weber, and the author for treating similar questions in the case of finitely many variables are extended, such that they apply to directional gradients. This approach also excludes zero-divisors for the considered elements in the finite Wigner chaos.
\end{abstract}

\maketitle


\section{Introduction}

\noindent
In 1998, P. Biane and R. Speicher established with their seminal work \cite{BianeSpeicher1998} a non-\linebreak commutative counterpart of classical stochastic calculus and Malliavin calculus in the realm of free probability. In particular, they introduced there the so-called (multiple) Wigner integrals
$$I^S_n(f) = \int_{\R_+^n} f(t_1,\dots,t_n)\ dS_{t_1} \cdots dS_{t_n}$$
for $f\in L^2(\R_+^n)$ on $\R_+ = [0,\infty)$ as the free counterpart of the classical (multiple) Wiener-It\^o integrals \cite{Wiener1938,Ito1951,Ito1952}. Despite some clear peculiarities of these free objects, their construction proceeds to a great extend parallel to the classical case, roughly speaking by replacing the classical Brownian motion by its free relative $(S_t)_{t\geq0}$. In analogy to the classical Wiener-It\^o chaos, these Wigner integrals form the so-called Wigner chaos, which likewise enjoys many properties similar to the classical Wiener-It\^o chaos; e.g. \cite{KempNourdinPeccatiSpeicher2012}.

We point out that the increments of the free Brownian motion $(S_t)_{t\geq0}$ carry the semicircular distribution as the free equivalent of the normal distribution from classical probability theory. It might seem strange at first sight that the nomenclature of Wigner integrals refers explicitly to Wigner, although his work clearly predates the birth of free stochastic calculus. However, this simply highlights the very important fact that the semicircular distribution already appeared in Wigner's famous semicircle law and that this rather surprising connection to random matrix theory, which was later clarified by Voiculescu, marks the starting point of an extremely fruitful interaction between random matrix theory and the theory of operator algebras.

Classical Malliavin calculus has many important applications (cf. \cite{Nualart2006, Nualart2009}). In particular, it became prominent for its use in treating regularity questions in different situations, as e.g. for distributions of random variables in the Wiener-It\^o chaos. For instance, it was used by Shigekawa \cite{Shigekawa1980} for proving that any non-trivial element in the finite Wiener-It\^o chaos, i.e. any non-constant finite sum of Wiener-It\^o integrals, has a distribution which is absolutely continuous with respect to the Lebesgue measure.

In contrast, in the world of free probability, distributions of non-commutative random variables that appear in the Wigner chaos are poorly understood. The aim of this paper is a first step towards a better understanding of these distributions by answering one of the fundamental questions formulated by Nourdin and Peccati in \cite[Remark 1.6]{NourdinPeccati2013}, namely: can the distribution of any non-constant self-adjoint Wigner integral have atoms or not? We will see that the answer to this question is no in full generality. Even more, we will show that the distribution of self-adjoint elements in the finite Wigner chaos, i.e. non-commutative random variables of the form
$$I_1^S(f) + I^S_2(f_2) + \dots + I_N^S(f_N)$$
with mirror-symmetric $f_n \in L^2(\R_+^n)$ for $n=1,\dots,N$ and $f_N\neq 0$, cannot have atoms. This is the content of of our main Theorem \ref{MainThm}.

Although this result is clearly in accord with the classical result of Shigekawa \cite{Shigekawa1980}, the proof of Theorem \ref{MainThm} uses completely different methods. Shigekawa's approach is based on arguments which are specially adapted to the commutative setting. In fact, he uses Malliavin's Lemma, which is a powerful result that provides a sufficient condition for a measure on $\R^d$ to be absolutely continuous with respect to Lebesgue measure. The non-commutativity in our situation forces us therefore to follow a totally different strategy, which is inspired by recently developed methods \cite{MaiSpeicherWeber2015,Shlyakhtenko2014}.

In free probability, regularity questions of this type were successfully addressed only quite recently \cite{ShlyakhtenkoSkoufranis2013,MaiSpeicherWeber2014,Shlyakhtenko2014,MaiSpeicherWeber2015,CharlesworthShlyakhtenko2015}. Our considerations here are very much based on the paper \cite{MaiSpeicherWeber2015}, where it was shown that in a von Neumann algebra $M$, which is endowed with a faithful normal tracial state $\tau$, the distribution of any non-constant self-adjoint polynomial expression $P(X_1,\dots,X_n)$ in finitely many self-adjoint variables $X_1,\dots,X_n\in M$ does not have atoms if the so-called non-microstates free entropy dimension $\delta^\ast(X_1,\dots,X_n)$ is maximal, i.e., if it satisfies the condition $\delta^\ast(X_1,\dots,X_n)=n$.

We note that the quantity $\delta^\ast(X_1,\dots,X_n)$ has its origin among other important quantities in the work of Voiculescu. He transferred in a groundbreaking series of papers \cite{Voi-Entropy-I, Voi-Entropy-II, Voi-Entropy-III, Voi-Entropy-IV, Voi-Entropy-V, Voi-Entropy-VI} (see also the survey article \cite{Voi-Entropy-Surv}) the classical notions of entropy and Fisher information to the non-commutative world. At the base of our work are techniques from the so-called non-microstates approach presented in \cite{Voi-Entropy-V, Voi-Entropy-VI}.

Formulated in general terms, so that it can be applied in our situation, the method of \cite{MaiSpeicherWeber2015} works as follows:
\begin{enumerate}
 \item \label{item:step1} Rephrase the question of absence of atoms in more algebraic terms as a question about the absence of (certain) zero-divisors.
 \item \label{item:step2} Prove that zero-divisors survive under special operations that are built on non-commutative derivations. This means that zero-divisors for some particular non-commutative random variable induce zero-divisors for some other non-commutative random variables of ``lower degree'', where the term ``degree'' refers to the degree of the considered polynomial, or in general to some natural grading on the space of non-commutative random variables under consideration.
 \item \label{item step3} Iterate the procedure of \eqref{item:step2} until reaching a non-commutative random variable of degree zero and check that the obtained element cannot be zero under the imposed conditions on the initial non-commutative random variable. This will lead to a contradiction and hence excludes zero-divisors. 
\end{enumerate}

It might be of independent interest that Step \eqref{item:step1} establishes a very interesting relationship to the work of Linnell \cite{Linnell1991, Linnell1992, Linnell1993, Linnell1998} on analytic versions of the zero divisor conjecture, particularly in the case of the free group. In fact, we will prove the more general statement that the product of any non-commutative random variable in the finite Wigner chaos, which is non-zero, with any non-zero element from the von Neumann algebra generated by the underlying free Brownian motion cannot be zero as well.

The crucial part is Step \eqref{item:step2}, which relies in \cite{MaiSpeicherWeber2015} as well as in our considerations heavily on results of Dabrowski \cite{Dab-Gamma, Dab-free_stochastic_PDE}, concerning bounds for the non-commutative derivatives that underlie the non-microstates approach to free Fisher information and free entropy of \cite{Voi-Entropy-IV} and also for more general derivations.

In contrast to the preceding studies, which especially concern the case of finitely many variables, the underlying von Neumann algebra in the setting of Wigner integrals is generated by a free Brownian motion $(S_t)_{t\geq0}$ and therefore by an uncountable family of semicircular elements, indexed by the continuous parameter $t\geq0$. Accordingly, the role of non-commutative derivatives in \cite{MaiSpeicherWeber2014,MaiSpeicherWeber2015} is taken over here by the directional gradient operators of free Malliavin calculus. Thus, the subsequent investigations can be seen as a continuous extension of the previous work \cite{ShlyakhtenkoSkoufranis2013,MaiSpeicherWeber2014,Shlyakhtenko2014,MaiSpeicherWeber2015,CharlesworthShlyakhtenko2015}.

In \cite{MaiSpeicherWeber2014}, which is an earlier version of \cite{MaiSpeicherWeber2015}, the absence of atoms in the distribution of $P(X_1,\dots,X_n)$ for non-constant self-adjoint polynomials $P$ was first shown under the stronger assumption of finite non-microstates free Fisher information $\Phi^\ast(X_1,\dots,X_n)$. Based on these ideas, Shlyakhtenko \cite{Shlyakhtenko2014} was able to prove a significant extension, namely to the most general case of full non-microstates entropy dimension $\delta^\ast(X_1,\dots,X_n)$, by involving different techniques from \cite{ConnesShlyakhtenko2005}. However, shortly after \cite{Shlyakhtenko2014}, the authors of \cite{MaiSpeicherWeber2014} were also able to upgrade their own methods to this generality, which led to the final version \cite{MaiSpeicherWeber2015}.

Deep results of Shlyakhtenko and Skoufranis \cite{ShlyakhtenkoSkoufranis2013} characterize the possible sizes of atoms that can appear in distributions of polynomial expressions $P(X_1,\dots,X_n)$ in non-commutative random variables $X_1,\dots,X_n$, which have not necessarily non-atomic distributions, (and even more matrices $(P_{ij}(X_1,\dots,X_n))_{i,j=1}^d$ thereof) under the assumption that $X_1,\dots,X_n$ are freely independent.
Since the non-microstates free entropy is additive for freely independent variables and since in the case of a single self-adjoint variable $X$ the maximality condition $\delta^\ast(X) = 1$ holds if and only if the distribution of $X$ has no atomic part, the results from \cite{MaiSpeicherWeber2015,Shlyakhtenko2014} clearly generalize some parts of the statements given in \cite{ShlyakhtenkoSkoufranis2013}. However, the full range of regularity results presented in \cite{ShlyakhtenkoSkoufranis2013} is still out of reach in this generality, but nevertheless, one expects that indeed for most of these properties rather the maximality of the non-microstates free entropy dimension matters than the free independence of the involved variables.

We point out that certain questions concerning the non-singularity and absolute continuity of distributions were addressed recently by Charlesworth and Shlyakhtenko \cite{CharlesworthShlyakhtenko2015}, in continuation of \cite{Shlyakhtenko2014}.

The paper is organized as follows.
In Section \ref{sec:Wigner-regularity}, we state our main result Theorem \ref{MainThm} on the regularity of distributions of Wigner integrals. For reader's convenience, we recall there also the fundamental definition of a free Brownian motion and the construction of Wigner integrals, as it can be found in the seminal work \cite{BianeSpeicher1998}. This exposition of the foundations of free stochastic calculus will then be continued in Section \ref{sec:StochasticCalculus}. In particular, we will define there the main operators of free Malliavin calculus and collect some results from \cite{BianeSpeicher1998}, which will be used later on.
Section \ref{sec:Derivations} is then devoted to the theory of non-commutative derivations. At first, we will put several results from \cite{Voi-Entropy-V} and \cite{Dab-Gamma} (see also \cite{Dab-free_stochastic_PDE}) in a uniform framework. Based on this, we will then obtain a significant generalization of a result that was obtained in \cite{MaiSpeicherWeber2015}, namely Proposition \ref{prop:key-inequality}, which is at the core of Step \eqref{item:step2} and hence crucial for the proof of Theorem \ref{MainThm}.
Finally, in Section \ref{sec:ProofMainThm}, we will piece together these ingredients for the actual proof of Theorem \ref{MainThm}. For this purpose, we will introduce the notion of directional gradients. The proof itself relies then on the fact that directional gradients, which belong by definition to free Malliavin calculus as presented Section \ref{sec:StochasticCalculus}, fit also nicely into the general framework of non-commutative derivations as considered in Section \ref{sec:Derivations}. Indeed, this will allow us to follow the aforementioned strategy in the spirit of \cite{MaiSpeicherWeber2014, MaiSpeicherWeber2015}.

\tableofcontents

\section{Wigner integrals and regularity of their distributions}
\label{sec:Wigner-regularity}

In this section, we provide all basic terminology and background knowledge as far as it is needed for stating our main result, Theorem \ref{MainThm}.

First of all, we will briefly recall some very basic concepts of free probability, before we proceed by giving the definition of a free Brownian motion and by presenting the construction of free Wigner integrals as they were introduced by Biane and Speicher in \cite{BianeSpeicher1998}; see also \cite{Speicher2003} and \cite{KempNourdinPeccatiSpeicher2012}.

The introduction to free stochastic calculus will be continued later in Section \ref{sec:StochasticCalculus}.

At the beginning, a few words on tensor products are in order. Throughout the paper, tensor products are understood as tensor products over the complex numbers $\C$. Moreover, we lay down here that the purely algebraic tensor product of complex vector spaces or complex algebras will be denoted by $\odot$, whereas the more familiar symbol $\otimes$ is reserved for its ``natural'' closure in the corresponding analytic setting, as for instance for Hilbert spaces or von Neumann algebras. Since the tensor sign will appear mostly in its closed version, this convention saves us from decorating the tensor signs repeatedly with fancy tags and hence keeps the notation as simple as possible.

\subsection{Non-commutative probability spaces and distributions}

The actual amount of techniques from free probability theory that are needed explicitly below is surprisingly small. The reason is that they are mostly hidden in the quoted results from free stochastic calculus and thus the computations involving them are just outsourced to other papers. However, we prefer to give a separate introduction to the very basic concepts of free probability theory, since it supplies the right language for our considerations.

Any reader, who is interested in a more detailed introduction to free probability theory, is cordially invited to have a look at \cite{VoiculescuDykemaNica1992}, \cite{Voiculescu2000}, or \cite{NicaSpeicher2006} for instance.

At the basis of free probability are non-commutative probability spaces. A \emph{non-commutative probability space $(\A,\phi)$} consists of a unital complex algebra $\A$ and a linear functional $\phi: \A\to\C$ that satisfies $\phi(1)=1$. Referring to classical probability theory, elements of $\A$ are called \emph{non-commutative random variables} and $\phi$ is called \emph{expectation}.

This nomenclature is justified by the observation that any classical probability space $(\Omega,\Sigma,\P)$ induces by $\A = L^\infty(\Omega,\Sigma,\P)$ and $\phi(X) = \int_\Omega X(\omega) d\P(\omega)$ a standard example of a non-commutative probability space (which is actually commutative).

In the generality of this purely algebraic setting, we can already introduce the notion of free independence. Unital subalgebras $(\A_i)_{i\in I}$ of $\A$ are called \emph{freely independent} (or just \emph{free}), if for any choice of finitely many indices $i_1,\dots,i_n\in I$, $n\in\N$, satisfying $i_1\neq i_2, i_2 \neq i_3,\dots, i_{n-1} \neq i_n$, and for any choice of elements $X_k \in \A_{i_k}$ with $\phi(X_k)=0$ for $k=1,\dots,n$, the condition $\phi(X_1 \cdots X_n) = 0$ is fulfilled.
Consequently, we call non-commutative random variables $(X_i)_{i\in I}$ in $\A$ \emph{freely independent} (or just \emph{free}), if the subalgebras $(\A_i)_{i\in I}$ are freely independent, where $\A_i$ denotes for each $i\in I$ the unital subalgebra of $\A$ that is generated by $X_i$. 

Roughly speaking, free independence provides a rule to calculate mixed moments. For any tuple $(X_1,\dots,X_n)$ of non-commutative random variables in $\A$, we refer to the collection of all \emph{moments}
$$\phi(X_{i_1} X_{i_2} \cdots X_{i_k}),\qquad k\geq 0,\ 1\leq i_1,\dots,i_k \leq n$$
(including also the trivial moment $\phi(1)=1$) as the \emph{(joint) distribution $\mu_{X_1,\dots,X_n}$ of $(X_1,\dots,X_n)$}. If the non-commutative random variables $X_1,\dots,X_n$ are freely independent, then the distribution $\mu_{X_1,\dots,X_n}$ is completely determined by the single variable distributions $\mu_{X_1},\dots,\mu_{X_n}$. For the seek of completeness, we point out that one can make this relation much more explicit by using the powerful combinatorial concept of \emph{free cumulants} as it was introduced to free probability by Speicher.

It was a fundamental observation of Voiculescu that the distribution $\mu_{X+Y}$ of two freely independent non-commutative random variables $X$ and $Y$ depends only on the distributions $\mu_X$ and $\mu_Y$ of $X$ and $Y$, respectively, and not on the concrete realization of $X$ and $Y$. Consequently, he defined the \emph{free additive convolution $\boxplus$} on abstract distributions by $\mu_X \boxplus \mu_Y := \mu_{X+Y}$.

For our purposes, it is necessary to impose some additional analytic structure. If we consider a \emph{$C^\ast$-probability space $(\A,\phi)$}, i.e. a non-commutative probability space $(\A,\phi)$, where $\A$ is a unital $C^\ast$-algebra and $\phi$ a state on $\A$, then the distribution $\mu_X$ of any self-adjoint non-commutative random variable $X$ in $\A$ can be identified with the compactly supported Borel probability measure $\mu_X$ on the real line $\R$ that is uniquely determined by the condition
$$\phi(X^k) = \int_\R t^k\, d\mu_X(t) \qquad\text{for $k=0,1,2,\dots$}.$$
Accordingly, the free additive convolution $\boxplus$ gives rise to a binary operation on all (compactly supported) Borel probability measures on $\R$. 

Here, we will mainly work in the setting of tracial $W^\ast$-probability spaces. A \emph{tracial $W^\ast$-probability space $(M,\tau)$} means a non-commutative probability space $(M,\tau)$, where $M$ is a von Neumann algebra and $\tau$ a faithful normal tracial state on $M$.

If $(M_1,\tau_1)$ and $(M_2,\tau_2)$ are two tracial $W^\ast$-probability spaces, then also their von Neumann algebra tensor product $M_1 \otimes M_2$ becomes, endowed with the tensor product state $\tau_1\otimes\tau_2$, a tracial $W^\ast$-probability space.

Another construction that will be used repeatedly in the subsequent considerations are the non-commutative $L^p$-spaces. Given any tracial $W^\ast$-probability space $(M,\tau)$, we may introduce the non-commutative $L^p$-spaces $L^p(M,\tau)$ for $1\leq p \leq \infty$ as the completion of $M$ with respect to the norm $\|x\|_{L^p(M,\tau)} := \tau\big((x^\ast x)^{\frac{p}{2}}\big)^{\frac{1}{p}}$, and for $p=\infty$ simply by $L^\infty(M,\tau) := M$ where we put $\|x\|_{L^\infty(M,\tau)} := \|x\|$. Whenever it is not necessary to indicate explicitly the underlying von Neumann algebra, we will abbreviate $\|\cdot\|_p := \|\cdot\|_{L^p(M,\tau)}$.

\subsection{Free Brownian motion}

Like the classical Brownian motion in the case of Wiener-It\^o integrals, the free Brownian motion is the fundamental object in free stochastic analysis and underlies in particular the construction of Wigner integrals. Thus, we want to recall now its definition.

Note that the definition itself will reflect the important fact that the role of the normal distribution in classical probability is taken over in free probability by the semicircular distribution as its free counterpart. We will denote by $\sigma_t$ the \emph{semicircular distribution with mean $0$ and variance $t>0$}, i.e. the compactly supported probability measure $\sigma_t$ on the real line $\R$ that is given by
$$d\sigma_t(x) = \frac{1}{2\pi t} \sqrt{4t-x^2}\, 1_{[-2\sqrt{t},2\sqrt{t}]}(x)\, dx.$$
Note that $(\sigma_t)_{t\geq 0}$ forms a semi-group with respect to the free additive convolution, i.e. we have that $\sigma_s \boxplus \sigma_t = \sigma_{s+t}$ holds for all $s,t\geq 0$.

\begin{definition}
Let $(M,\tau)$ be a tracial $W^\ast$-probability space. A family $(S_t)_{t \geq 0}$ of operators in $(M,\tau)$ is called \emph{free Brownian motion}, if there exists a \emph{filtration $(M_t)_{t\geq0}$ of $M$}, i.e. a family $(M_t)_{t\geq0}$ of von Neumann subalgebras $M_t$ of $M$ with
$$M_s \subseteq M_t \qquad\text{whenever $s\leq t$},$$
such that the following conditions are satisfied:
\begin{itemize}
\item We have $S_0=0$ and $S_t = S_t^\ast \in M_t$ for all $t\geq 0$.
\item For each $t > 0$, the distribution of $S_t$ is the semicircular distribution $\sigma_t$.
\item For all $0 \leq s < t$, the distribution of $S_t - S_s$ is the semicircular distribution $\sigma_{t-s}$. 
\item For all $0 \leq s < t$, the \emph{increment} $S_t-S_s$ is free from $M_s$, which means more precisely that the unital subalgebra generated by $S_t-S_s$ is free from $M_s$.
\end{itemize}
\end{definition}

A free Brownian motion can be constructed in several ways. For instance, one construction gives the free Brownian motion as the limit of matrix-valued classical Brownian motions as the dimension tends to infinity. In contrast to this certainly appealing but rather indirect approach, we will present in Subsection \ref{subsec:freeBrownianMotion} a construction of the free Brownian motion on the full Fock space over the Hilbert space $L^2(\R_+)$ of all square-integrable functions on the positive real half-line $\R_+ := [0,\infty)$. This has the advantage that it will not only prove the existence of the free Brownian motion but it will also give an additional structure to this important object, which is in fact the starting point of free Malliavin calculus. However, for the moment, we take the existence of a free Brownian motion for granted.

\subsection{Wigner integrals}

Presuming the existence of a free Brownian motion $(S_t)_{t\geq0}$ in a $W^\ast$-probability space $(M,\tau)$ with respect to a filtration $(M_t)_{t\geq0}$ of $M$, we may introduce now (multiple) Wigner integrals integrals with respect to $(S_t)_{t\geq0}$.

\begin{definition}
Let $n\in\N$ be given. We denote by $D^n \subset \R^n_+$ the collection of all diagonals in $\R^n_+$, i.e.
$$D^n := \{(t_1,\dots,t_n)\in\R^n_+|\ \text{$t_i = t_j$ for some $1\leq i,j\leq n$ with $i\neq j$}\}.$$

The construction of the \emph{(multiple) Wigner integral $I_n^S(f)$} for any function $f\in L^2(\R^n_+)$ proceeds as follows.
\begin{itemize}
 \item For any indicator function $f=1_E$ of some set
$$E = [s_1,t_1] \times \dots \times [s_n,t_n] \subset \R^n_+$$
that satisfies $E\cap D^n =\emptyset$, we define $I^S_n(f)$ by
$$I_n^S(f) = (S_{t_1} - S_{s_1}) \cdots (S_{t_n} - S_{s_n}).$$
 \item By linearity, we extend $I^S_n$ to all \emph{off-diagonal step functions}, i.e. to all step functions
$$f = \sum^m_{j=1} a_j 1_{E_j}$$
on $\R^n_+$, where each set $E_j\subset\R^n_+$ is of the form
$$E_j = [s_{j,1},t_{j,1}] \times \dots \times [s_{j,n},t_{j,n}]$$
and satisfies $E_j\cap D^n = \emptyset$.
 \item Since off-diagonal step functions are dense in $L^2(\R^n_+)$ (an important fact, which is actually not hard to prove, but which is definitely worth to think about for a moment) and since the \emph{It\^o isometry}
$$\tau(I^S_n(f)^\ast I^S_n(g)) = \langle g,f\rangle_{L^2(\R^n_+)}$$
holds for all off-diagonal step functions $f$ and $g$, we may finally extend $I^S_n$ isometrically to $L^2(\R_+^n)$.
\end{itemize}

For given $f\in L^2(\R^n_+)$, we will write
$$I_n^S(f) = \int_{\R_+^n} f(t_1, \dots, t_n)\, dS_{t_1} \cdots dS_{t_n}.$$
\end{definition}

Note that multiple Wigner integrals $I^S_n(f)$ are for general $f\in L^2(\R^n_+)$ by definition elements of $L^2(M,\tau)$. But in fact, it turns out that $I_n^S(f)$ belongs to $M$ for each $f\in L^2(\R^n_+)$ (and actually, to be more precise, it belongs to the $C^\ast$-subalgebra of $M$ that is generated by the free Brownian motion $(S_t)_{t\geq0}$). This is an immediate consequence of the fact that off-diagonal step functions are dense in $L^2(\R^n_+)$ and of \cite[Theorem 5.3.4]{BianeSpeicher1998}, which tells us that the operator norm can be bounded by a kind of Haagerup inequality, namely
\begin{equation}\label{Haagerup-inequality}
\Big\| \int_{\R^n_+} f(t_1,\dots,t_n)\, dS_{t_1} \cdots dS_{t_n}\Big\| \leq (n+1) \|f\|_{L^2(\R^n_+)} \qquad \text{for all $f\in L^2(\R^n_+)$}.
\end{equation}

Since Wigner integrals are bounded linear operators, we are of course allowed to multiply them, and it is therefore natural to ask, whether one can describe this operation also on the level of the corresponding functions. Indeed, this turns out to be possible and it leads to a free counterpart of It\^o's formula (see, for example, \cite[Theorem 2.11]{Speicher2003}). Although this result appears in many different formulations, it always reflects the same inherent structure that shows up, roughly speaking, under multiplication. We mention here the following version, which allows us to decompose products of Wigner integrals explicitly as linear combinations of Wigner integrals.

\begin{theorem}[Biane and Speicher, 1998, \cite{BianeSpeicher1998}]\label{thm:Wigner-products}
Let $f \in L^2(\R_+^n)$ and $g \in L^2(\R_+^m)$. For any $0\leq p \leq \min\{n,m\}$, we define the \emph{$p$'th contraction of $f$ and $g$} by
\begin{align*}
f \stackrel{p}{\smallfrown} g(t_1, \dots, t_{n+m-2p}) = \int_{\R_+^p} & f(t_1, \dots, t_{n-p}, s_1, \dots, s_p)\\
&  g(s_p, \dots, s_1, t_{n-p+1}, \dots, t_{n+m-2p})\, ds_1 \dots ds_p.
\end{align*}
Then the It\^o formula
$$I_n^S(f) I_m^S(g) = \sum_{p=0}^{\min\{n,m\}} I_{n+m-2p}^S(f \stackrel{p}{\smallfrown} g)$$
holds.
\end{theorem}

In principle, all previously collected facts about Wigner integrals put them in the most convenient setting of non-commutative probability, such that we can already talk about their (joint) distributions in a purely combinatorial sense. However, since we work here in the regular setting of $W^\ast$-probability spaces, we also want to study distributions of Wigner integrals in a stronger analytic sense, namely as (compactly supported) probability measures. Thus, we should have a criterion on the level of integrands that allows us to guarantee that the corresponding Wigner integral is self-adjoint. This criterion is provided by mirror symmetry.

It follows immediately from the definition of Wigner integrals that
$$I_n^S(f)^\ast = I^S_n(f^\ast) \qquad\text{for all $f\in L^2(\R_+^n)$}$$
holds, where the function $f^\ast \in L^2(\R_+^n)$ is determined for any $f\in L^2(\R_+^n)$ by
$$f^\ast(t_1,t_2,\dots,t_n) = \overline{f(t_n,\dots,t_2,t_1)}$$
for Lebesgue almost all $(t_1,\dots,t_n) \in \R_+^n$. As a consequence, any $f\in L^2(\R_+^n)$ satisfying $f=f^\ast$ gives a self-adjoint Wigner integral $I_n^S(f)$. We will call such $f\in L^2(\R_+^n)$ \emph{mirror symmetric}.

\subsection{Main Theorem}

Here, we are interested in properties of the distributions of Wigner integrals
$$I^S_n(f) = \int_{\R_+^n} f(t_1,\dots,t_n)\, dS_{t_1} \cdots dS_{t_n}$$
for mirror symmetric functions $f\in L^2(\R_+^n)$, and, more generally, in distributions of finite sums of such Wigner integrals like
$$Y = I_1^S(f_1) + I_2^S(f_2) + \dots + I_N^S(f_N)$$
for some $N\in\N$ and mirror symmetric functions $f_n\in L^2(\R_+^n)$ for $n=1,\dots,N$ with $f_N \neq 0$.

Surely one of the most basic questions one can ask about distributions in general is whether their support is connected or not. Basic functional analysis yields that this question can be reformulated in more operator algebraic terms to a question about the existence of non-trivial projections in the $C^\ast$-algebra that is generated by the considered operator.
Fortunately, this translation is also helpful in our situation: As we have mentioned above, Wigner integrals are in fact elements of the $C^\ast$-algebra that is generated by the free Brownian motion $(S_t)_{t\geq0}$. Hence, by quoting a results obtained by Guionnet and Shlyakhtenko in \cite{GuionnetShlyakhtenko2009}, which excludes non-trivial projections in $C^\ast(\{S_t|\ t\geq 0\})$, we may conclude without further effort that the distribution $\mu_Y$ of any operator $Y$ as above must have connected support.

However, apart from this observation, almost nothing was known until now about regularity properties of these distributions. In particular, as it was formulated by Nourdin and Peccati in \cite[Remark 1.6]{NourdinPeccati2013}, it remained an open questions whether the distribution of Wigner integrals of mirror symmetric functions being non-zero (except, of course, in the chaos of order zero) may have atoms or not. We are going to answer this question here by showing that the distribution of any such Wigner integral of a non-zero mirror symmetric function (and even of any non-constant finite sum of such Wigner integrals) does not have atoms.

Recall that an \emph{atom} of a Borel probability measure $\mu$ on $\R$ means some $\alpha\in\R$ satisfying the condition $\mu(\{\alpha\}) \neq 0$. 

The statement of the main theorem of this paper reads as follows.

\begin{theorem}\label{MainThm}
For given $N\in\N$, we consider mirror symmetric functions $f_n\in L^2(\R_+^n)$ for $n=1,\dots,N$, where we assume that $f_N \neq 0$. Then, the distribution $\mu_Y$ of
$$Y := I_1^S(f_1) + I_2^S(f_2) + \dots + I_N^S(f_N),$$
regarded as an element in $(M,\tau)$, has no atoms.
\end{theorem}

The proof of Theorem \ref{MainThm} will be given in Section \ref{sec:ProofMainThm}. We stress that the above statement clearly stays valid if we add to $Y$ a constant multiple of the identity. In fact, this will be a direct outcome of the proof of Theorem \ref{MainThm}, since we will use the chaos decomposition to deal with such shifts in a uniform way. More precisely, we can just encode constant multiples of the identity by the chaos of order zero.

Furthermore, we point out that Theorem \ref{MainThm} corresponds nicely to a classical result of Shigekawa \cite{Shigekawa1978, Shigekawa1980} (although its proof uses completely different methods for which there are by now no free analogues), which states that any non-trivial finite sum of Wiener-It\^o integrals has an absolutely continuous distribution, and hence cannot have atoms. Thus, confident of the far reaching parallelism between classical and free probability, we are tempted to conjecture in accordance with \cite{Speicher2013} that the analogy between Wiener-It\^o integrals and Wigner integrals goes even further, namely that any $Y$ like in Theorem \ref{MainThm} has in fact an absolutely continuous distribution. We leave this question to further investigations.

\section{Free stochastic calculus} 
\label{sec:StochasticCalculus}

One of the main pillars on which the proof of Theorem \ref{MainThm} rests is free stochastic calculus as it was introduced by Biane and Speicher in \cite{BianeSpeicher1998}. For readers convenience, we recall in this section the basic definitions and some results of this theory as far as necessary.

First of all, we will introduce the notion of biprocesses. Secondly, we will describe the concrete realization of the free Brownian motion on the full Fock space over $L^2(\R_+)$. This additional structure will finally allow us to introduce the basic operators of Malliavin calculus.

\subsection{Biprocesses}

We broach now the theory of biprocesses. Our exposition here heavily relies on \cite{BianeSpeicher1998}, \cite{Speicher2003}, and \cite{KempNourdinPeccatiSpeicher2012}.

Let us first introduce a few general notions. We denote by $\E(\R_+)$ the space of all complex valued functions $f$ on $\R_+$, which can be written as a finite sum
$$f = \sum^n_{j=1} a_j\, 1_{E_j}$$
for some intervals $E_1,\dots,E_n\subseteq\R_+$ of the form $E_j = [s_j,t_j)$ with $0\leq s_j < t_j < \infty$ for $j=1,\dots,n$ and complex numbers $a_1,\dots,a_n\in\C$. As usually, $1_E$ denotes the indicator function of a subset $E\subseteq\R_+$. It is easy to see that $\E(\R_+)$ is in fact a complex algebra.

For any unital complex algebra $\A$, the algebraic tensor product $\E(\R_+,\A) := \E(\R_+)\odot\A$ consists of all functions $f$ defined on $\R_+$ and taking values in $\A$, which can be written as
$$f = \sum^n_{j=1} A_j\, 1_{E_j}$$
for some intervals $E_1,\dots,E_n\subseteq\R_+$ of the form $E_j = [s_j,t_j)$ with $0\leq s_j < t_j < \infty$ for $j=1,\dots,n$ and elements $A_1,\dots,A_n\in\A$.

\subsubsection{Definition of biprocesses}

We are prepared now to define biprocesses. For the remaining part of this subsection, we fix a tracial $W^\ast$-probability space $(M,\tau)$ for which a filtration $(M_t)_{t\geq0}$ exists.

\begin{definition}
We distinguish several types of biprocesses, which are build on each other. Their definition proceeds as follows:
\begin{itemize}
 \item[(i)] The elements $$U:\ \R_+\to M \odot M,\ t\mapsto U_t$$ of $\E(\R_+,M \odot M)$ are called \emph{simple biprocesses}.
 \item[(ii)] A simple biprocess $U: \R_+ \to M \odot M$ is called \emph{adapted}, if the condition $U_t\in M_t \odot M_t$ is satisfied for all $t\geq0$. The set of all adapted simple biprocesses will be denoted by $\E^a(\R_+,M\odot M)$.
 \item[(iii)] We denote by $\B_p$ for $1\leq p \leq \infty$ the completion of $\E(\R_+,M \odot M)$, with respect to the norm $\|\cdot\|_{\B_p}$, which is given by $$\|U\|_{\B_p} := \bigg(\int_{\R_+} \|U_t\|^2_{L^p(M\otimes M,\tau\otimes\tau)}\, dt\bigg)^{\frac{1}{2}}.$$ An element of $\B_p$ is called an \emph{$L^p$-biprocess}.
 \item[(iv)] For $1\leq p \leq \infty$, the closure of $\E^a(\R_+, M\odot M)$ with respect to $\|\cdot\|_{\B_p}$ will be denoted by $\B^a_p$. Elements of $\B^a_p$ are called \emph{adapted $L^p$-biprocesses}.
\end{itemize}
\end{definition}

\subsubsection{Integration of biprocesses}

For our purposes, the integration theory of biprocesses is of great importance. We focus here first on the integration of $L^p$-biprocesses with respect to functions in $L^2(\R_+)$.

On the basic level of simple biprocesses, such integrals can be introduced quite easily: if $U$ is any simple biprocess, we may write
\begin{equation}\label{eq:simple-biprocess-repr} 
U = \sum^n_{j=1} U^{(j)}\, 1_{E_j}
\end{equation}
for some intervals $E_1,\dots,E_n\subseteq\R_+$ of the form $E_j = [s_j,t_j)$ with $0\leq s_j < t_j < \infty$ for $j=1,\dots,n$ and certain elements $U^{(1)},\dots,U^{(n)}\in M\odot M$. Then, we put
$$\int_{\R_+} U_t\, \overline{h(t)}\, dt := \sum^n_{j=1} \langle 1_{E_j}, h\rangle_{L^2(\R_+)}\, U^{(j)},$$
and it is easy to see that the value of this integral does not depend on the concrete choice of the representation \eqref{eq:simple-biprocess-repr}.

Sometimes, it is more appropriate to write a given simple biprocess $U$ in \emph{standard form}, i.e. in the form of \eqref{eq:simple-biprocess-repr}, where the intervals $E_1,\dots,E_n\subseteq\R_+$ are assumed to be pairwise disjoint.

By the construction presented above, we obtain a sesqui-linear pairing
$$\langle\cdot,\cdot\rangle:\ \E(\R_+, M \odot M) \times L^2(\R_+) \to M \odot M,$$
which is given by
$$\langle U,h \rangle := \int_{\R_+} U_t\, \overline{h(t)}\, dt$$
for any $U\in\E(\R_+,M\odot M)$ and $h\in L^2(\R_+)$. 

Since we want to extend $\langle\cdot,\cdot\rangle$ to a sesqui-linear paring between $\B_p$ and $L^2(\R_+)$, we need to study its continuity with respect to $\|\cdot\|_{\B_p}$. This will be done in the following lemma. In the case $p=\infty$, this property of $\langle\cdot,\cdot\rangle$ was already mentioned in \cite{BianeSpeicher1998}. The general case is probably also well-known to experts, but for the seek of completeness, we include here the straightforward proof.

\begin{lemma}\label{lem:biprocess-pairing}
Let $1\leq p \leq \infty$ be given. For any $U\in\E(\R_+,M\odot M)$ and $h\in L^2(\R_+)$, it holds true that
$$\|\langle U,h \rangle\|_{L^p(M\otimes M,\tau\otimes\tau)} \leq \|U\|_{\B_p} \|h\|_{L^2(\R_+)}.$$
\end{lemma}

\begin{proof}
Let $U\in\E(\R_+,M\odot M)$ and $h\in L^2(\R_+)$ be given and write $U$ in standard form
$$U = \sum^n_{j=1} U^{(j)}\, 1_{E_j}.$$
For any fixed $1\leq p \leq \infty$, we may check that
$$\|U\|_{\B_p} = \bigg(\int_{\R_+} \|U_t\|^2_{L^p(M\otimes M,\tau\otimes\tau)}, dt\bigg)^{\frac{1}{2}} = \bigg(\sum^n_{j=1} \lambda^1(E_j) \|U^{(j)}\|^2_{L^p(M\otimes M,\tau\otimes\tau)}\bigg)^{\frac{1}{2}},$$
where $\lambda^1$ denotes the Lebesgue measure on $\R$. Thus, applying twice the Cauchy-Schwarz inequality yields as desired
\begin{align*}
\|\langle U,h\rangle&\|_{L^p(M\otimes M,\tau\otimes\tau)}\\
&\leq \sum^n_{j=1} |\langle 1_{E_j}, h\rangle_{L^2(\R_+)}|\, \|U^{(j)}\|_{L^p(M\otimes M,\tau\otimes\tau)}\\
&= \sum^n_{j=1} |\langle 1_{E_j}, 1_{E_j} h\rangle_{L^2(\R_+)}|\, \|U^{(j)}\|_{L^p(M\otimes M,\tau\otimes\tau)}\\
&\leq \sum^n_{j=1} \|1_{E_j} h\|_{L^2(\R_+)} \|1_{E_j}\|_{L^2(\R_+)} \|U^{(j)}\|_{L^p(M\otimes M,\tau\otimes\tau)}\\
&\leq \bigg(\sum^n_{j=1} \|1_{E_j} h\|_{L^2(\R_+)}^2\bigg)^{\frac{1}{2}} \bigg(\sum^n_{j=1} \|1_{E_j}\|_{L^2(\R_+)}^2 \|U^{(j)}\|_{L^p(M\otimes M,\tau\otimes\tau)}^2\bigg)^{\frac{1}{2}}\\
&\leq \|h\|_{L^2(\R_+)} \|U\|_{\B_p},                               
\end{align*}
where we used in addition that due to the pairwise orthogonality of the functions $\{1_{E_j} h|\ j=1,\dots,n\}$
$$\bigg(\sum^n_{j=1} \|1_{E_j} h\|_{L^2(\R_+)}^2\bigg)^{\frac{1}{2}} \leq \|h\|_{L^2(\R_+)}$$
holds and that we have $\|1_E\|_{L^2(\R_+)}^2 = \lambda^1(E)$ for any Borel set $E\subseteq \R_+$ with finite Lebesgue measure.
\end{proof}

Due to the inequality that we have established in Lemma \ref{lem:biprocess-pairing}, the definition of $\langle\cdot,\cdot\rangle$ extends now naturally to $\B_p$.

\begin{definition}\label{def:biprocess-pairing}
For any $1\leq p \leq \infty$, the sesqui-linear pairing
$$\langle\cdot,\cdot\rangle:\ \E(\R_+, M \odot M) \times L^2(\R_+) \to M \odot M,\qquad \langle U,h\rangle = \int_{\R_+} U_t \overline{h(t)}\, dt,$$
extends continuously according to
$$\|\langle U,h \rangle\|_{L^p(M\otimes M,\tau\otimes\tau)} \leq \|U\|_{\B_p} \|h\|_{L^2(\R_+)}.$$
to a sesqui-linear pairing
$$\langle\cdot,\cdot\rangle:\ \B_p \times L^2(\R_+) \to L^p(M \otimes M, \tau\otimes\tau).$$
\end{definition}

\subsubsection{Stochastic integrals of biprocesses}

Next, we are going to define stochastic integrals $\int_{\R_+} U_t\sharp dS_t$ of biprocesses $U$ with respect to the free Brownian motion $(S_t)_{t\geq0}$. For this purpose, we first have to introduce the notation $\sharp$, which appears here and repeatedly in the non-commutative setting.

\begin{remark}\label{rem:sharp}
Let $\A$ and $\B$ be complex algebras. If $\M$ is an $\A$-$\B$-bimodule, then we denote by $\sharp$ the operation $(\A \odot \B) \times \M \to \M$ that is determined by linear extension of $(a \otimes b) \sharp m := a \cdot m \cdot b$. Even more, if we would replace here $\B$ by its opposite algebra $\B^\op$, then $\sharp$ would give rise to a left action of the algebraic tensor product $\A \odot \B^\op$ on $\M$. But since the multiplicative structure of $\A \odot \B$ will play a minor role in our considerations, we will not care about this subtlety in the following.
\end{remark}

\begin{definition}
Let $(S_t)_{t\geq0}$ be a free Brownian motion in $M$ with respect to its given filtration $(M_t)_{t\geq0}$.
\begin{itemize}
 \item For any simple biprocess $U\in\E(\R_+,M\odot M)$, we define $$\int_{\R_+} U_t \sharp dS_t := \sum^n_{j=1} U^{(j)} \sharp (S_{t_j} - S_{s_j}) = \sum^n_{j=1} \sum^{m_j}_{i=1} A_i^{(j)} (S_{t_j} - S_{s_j}) B_i^{(j)},$$
where $U$ is written in the form \eqref{eq:simple-biprocess-repr} for intervals $E_j = [s_j,t_j)$ with $0 \leq s_j < t_j < \infty$ and elements $U^{(j)}\in M \odot M$ of the form $U^{(j)} = \sum^{m_j}_{i=1} A_i^{(j)} \otimes B_i^{(j)}$ for $j=1,\dots,n$.
 \item If $U,V \in\E^a(\R_+, M \odot M)$ are simple adapted biprocesses, then the \emph{general Wigner-It\^o isometry} (cf. \cite[Proposition 2.7]{Speicher2003}) tells us that $$\langle \int_{\R_+} U_t\sharp dS_t, \int_{\R_+} V_t\sharp dS_t \rangle = \int_{\R_+} \langle U_t, V_t\rangle\, dt =: \langle U, V\rangle_{\B_2}$$ holds. Thus, we have in particular that $$\Big\|\int_{\R_+} U_t\sharp dS_t\Big\|_2 = \|U\|_{\B_2}$$ for all $U\in \E^a(\R_+, M\odot M)$. Therefore, the integral $\int_{\R_+} U_t\sharp dS_t$ extends from simple adapted biprocesses to any adapted $L^2$-biprocess $U\in\B_2^a$ in such a way that the induced mapping
$$U \mapsto \int_{\R_+} U_t\sharp dS_t$$ is isometric from $\B_2^a$ to $L^2(M,\tau)$.
\end{itemize}
\end{definition}

\subsection{The free Brownian motion on the full Fock space}\label{subsec:freeBrownianMotion}

We come back now to the construction of the free Brownian motion. As we announced earlier, we will do this here in an explicit way on the full Fock space over $L^2(\R_+)$. These techniques will be used to build up free Malliavin calculus, in the same way as classical Malliavin calculus is built on the symmetric Fock space.

\subsubsection{The full Fock space and field operators}

We first recall the construction of the full Fock space over an arbitrary complex Hilbert space.

Recall that in the context of complex Hilbert spaces, the symbol $\odot$ stands for the algebraic tensor product (over the complex numbers $\C$), whereas its completion with respect to the canonical inner product will be denoted by $\otimes$.

\begin{definition}
Let $(H,\langle\cdot,\cdot\rangle_H)$ be a complex Hilbert space. We define the \emph{full Fock space} $\F(H)$ associated to $H$ as the complex Hilbert space that is given by
$$\F(H) := \bigoplus^\infty_{n=0} H^{\otimes n},$$
where $\bigoplus$ is understood as Hilbert space operation. Therein, we declare that $H^{\otimes 0} := \C \Omega$ for some fixed vector $\Omega$ of norm $1$, which we call the \emph{vacuum vector} of $\F(H)$.

More explicitly, the inner product $\langle\cdot,\cdot\rangle$ on $\F(H)$ is determined by the following rules: We have
$$\langle g_1 \otimes \dots \otimes g_m, h_1 \otimes \dots \otimes h_n\rangle = 0 \qquad \text{if $m\neq n$}$$
and in the case $m=n$
$$\langle g_1 \otimes \dots \otimes g_m, h_1 \otimes \dots \otimes h_m\rangle = \langle g_1,h_1 \rangle_H \cdots \langle g_m,h_m \rangle_H.$$
\end{definition}

Later on, we will also work with some special (non-closed) subspaces of the full Fock space $\F(H)$, involving an infinite but algebraic direct sum, namely
\begin{itemize}
 \item $\displaystyle{\F_\alg(H) := \sideset{}{_\alg}\bigoplus_{n=0}^\infty H^{\odot n}}$, i.e. the subspace of $\F(H)$ that consists of finite sums of tensor products of vectors in $H$, and
 \item $\displaystyle{\F_\fin(H) := \sideset{}{_\alg}\bigoplus_{n=0}^\infty H^{\otimes n}}$, i.e. the subspace of $\F(H)$ that consists of finite sums of elements in the Hilbert spaces $H^{\otimes n}$.
\end{itemize}

It is clear by definition that we have the inclusions $\F_\alg(H) \subseteq \F_\fin(H) \subseteq \F(H)$ and that both subspaces $\F_\alg(H)$ and $\F_\fin(H)$ are dense in $\F(H)$.

On the full Fock space $\F(H)$, we may introduce the so-called field operators. In the case $H=L^2(\R_+)$, these operators will provide the desired realization of the free Brownian motion.

\begin{definition}
Let $(H,\langle\cdot,\cdot\rangle_H)$ be a complex Hilbert space. For each $h\in H$ we introduce the following operators on the full Fock space $\F(H)$ over $H$:
\begin{itemize}
 \item[(i)] The \emph{creation operator} $l(h) \in B(\F(H))$ is determined by
            \begin{eqnarray*}
            l(h)\, h_1 \otimes \dots \otimes h_n &=& h \otimes h_1 \otimes \dots \otimes h_n,\\
            l(h)\, \Omega &=& h.
            \end{eqnarray*}
 \item[(ii)] The \emph{annihilation operator} $l^\ast(h) \in B(\F(H))$ is given by
            \begin{eqnarray*}
            l^\ast(h)\, h_1 \otimes \dots \otimes h_n &=& \langle h,h_1\rangle_H h_2 \otimes \dots \otimes h_n,\qquad n\geq 2,\\
            l^\ast(h)\, h_1 &=& \langle h, h_1\rangle_H \Omega,\\
            l^\ast(h)\, \Omega &=& 0.
            \end{eqnarray*}
 \item[(iii)] The \emph{field operator} $X(h) \in B(\F(H))$ is defined by $$X(h) := l(h) + l^\ast(h).$$
\end{itemize}
\end{definition}

An easy calculation shows that we have $l^\ast(h)=l(h)^\ast$ for all $h\in H$, as the notation suggests. As an immediate consequence, $X(h) = X(h)^\ast$ holds for each $h\in H$.

In order to obtain a $W^\ast$-probability space, in which the free Brownian lives, it is natural to consider the von Neumann algebra generated by field operators $X(h)$ for a sufficiently large family of vectors $h$. As it turns out, the right choice for this purpose are the ``real'' vectors $h$. More formally, we will consider the full Fock space over the \emph{complexification} $H_\C = H \oplus i H$ of any real Hilbert space $(H,\langle\cdot,\cdot\rangle_H)$. The ``real'' vectors are then naturally those, which are coming from $H$. We shall make this more precise with the following definition.

\begin{definition}
Let $H$ be a real Hilbert space and denote by $H_\C = H \oplus i H$ its complexification. We define the von Neumann algebra $\S(H) \subseteq B(\F(H_\C))$ by
$$\S(H) = \vN\big(\{X(h)|\ h \in H\}\big).$$
We may endow $\S(H)$ with the \emph{vacuum expectation} $\tau: \S(H) \rightarrow \C$ given by
$$\tau(X) = \langle X\Omega,\Omega \rangle.$$
\end{definition}

Due to the fact that $H$ is a real Hilbert space, we are in the nice situation that $\tau$ gives a faithful normal tracial state on $\S(H)$. Thus, we have obtained a $W^\ast$-probability space $(\S(H),\tau)$.

Later on, we will also use the unital $\ast$-algebra $\S_\alg(H)$ that is given by
$$\S_\alg(H) : = \alg\big(\{X(h)|\ h \in H\}\big).$$
Clearly, $\S_\alg(H) \subseteq \S(H) \subseteq B(\F(H_\C))$.

It is a very nice feature of $(\S(H),\tau)$ that its $L^2$-space $L^2(\S(H),\tau)$ can be identified in a natural way with the corresponding full Fock space $\F(H_\C)$. This important observation is at the base of free Malliavin calculus.

Since we have for all $X_1,X_2\in\S(H)$ that
$$\langle X_1,X_2\rangle_{L^2(\S(H),\tau)} = \tau(X_2^\ast X_1) = \langle (X_2^\ast X_1) \Omega,\Omega\rangle_{\F(H)} = \langle X_1 \Omega, X_2 \Omega\rangle_{\F(H)},$$
we see that the map
$$\Phi_0:\ \S(H) \to \F(H_\C),\ X \mapsto X\Omega$$
admits an isometric extension
$$\Phi:\ L^2(\S(H),\tau) \to \F(H_\C).$$
The following lemma allows us to conclude that $\Phi$ is even more surjective and hence gives the desired isometric isomorphism between $L^2(\S(H),\tau)$ and $\F(H_\C)$. A proof can be found in \cite[Section 5.1]{BianeSpeicher1998}.

\begin{lemma}\label{lem:Wick-product}
Given $h_1, \dots, h_n \in H_\C$, then  there exists a unique operator
$$W(h_1 \otimes \dots \otimes h_n) \in \S(H),$$
called the \emph{Wick product of $h_1 \otimes \dots \otimes h_n$}, such that
$$W(h_1 \otimes \dots \otimes h_n) \Omega = h_1 \otimes \dots \otimes h_n.$$
More precisely, if $(e_j)_{j \in J}$ is an orthonormal basis of $H$ then
$$W(e_{j_1}^{\otimes k_1} \otimes \dots \otimes e_{j_n}^{\otimes k_n}) = U_{k_1}(X(e_{j_1})) \cdots U_{k_n}(X(e_{j_n})),$$
where $j_1 \neq j_2 \neq \dots \neq j_n$ and $U_k$ denotes the \emph{$k$'th (normalized) Chebyshev polynomial of the second kind}. These polynomials are determined by $U_0(X) = 1$, $U_1(X)=X$ and the recursion $U_{k+1}(X) = X U_k(X) - U_{k-1}(X)$ for $k\geq 1$.
\end{lemma}

Note that the lemma implies in particular that $\Phi_0(\S_\alg(H)) = \F_\alg(H)$.

\subsubsection{$\F(L^2(\R_+))$ and the free Brownian motion}

We return now to the actual goal of this subsection, namely the construction of the free Brownian motion. This is achieved by applying the foregoing constructions to the real Hilbert $H = L^2(\R_+,\R)$, whose complexification is clearly given by $H_\C \cong L^2(\R_+)$.

In the $W^\ast$-probability space $(\S,\tau)$ where we abbreviate $\S:=\S(L^2(\R_+,\R))$, the free Brownian motion $(S_t)_{t\geq 0}$ is obtained by putting
$$S_t := X(1_{[0,t]}) \qquad\text{for all $t \geq 0$}.$$
The corresponding filtration $(\S_t)_{t\geq0}$ of $\S$ is given by
$$\S_t := \vN\big(\{X(h)|\ h\in L^2([0,t],\R)\}\big),$$
where we regard $L^2([0,t],\R)$ as a subspace of $L^2(\R_+,\R)$ via extension by zero. In fact, $\S_t$ is generated as a von Neumann algebra by $\{S_s|\ 0 \leq s \leq t\}$, while $\S$ is generated by $\{S_s|\ s\geq 0\}$.

The very concrete realization of the free Brownian motion in the $W^\ast$-probability space $(\S,\tau)$ has the advantage that it carries the rich structure provided by the underlying Fock space $\F := \F(L^2(\R_+))$ by the isometric isomorphism
$$\Phi:\ L^2(\S,\tau) \to \F,$$
which was obtained by isometric extension of the map $\Phi_0: \S \to \F$ given by $\Phi_0(X) = X\Omega$. This will be used in the next subsection on free Malliavin calculus.

But before continuing in this direction, we first discuss the chaos decomposition for arbitrary elements in $L^2(\S,\tau)$, which emerges from the isomorphism $\Phi$. In the simplest case, it boils down to a nice relation between Wigner integrals and the Wick products as introduced in Lemma \ref{lem:Wick-product}. More precisely, we have for all $h_1,\dots,h_n\in L^2(\R_+)$ that
$$W(h_1 \otimes \dots \otimes h_n) = I_n^S(h_1 \otimes \dots \otimes h_n) = \int_{\R_+^n} h_1(t_1) \cdots h_n(t_n)\, dS_{t_1} \cdots dS_{t_n}.$$
This observation is generalized by the following result.

\begin{proposition}[Proposition 5.3.2. in \cite{BianeSpeicher1998}]
The inverse of the isomorphism $\Phi: L^2(\S,\tau) \to \F$ is given by
$$I^S:\ \F \to L^2(\S,\tau),\ f\mapsto I^S(f),$$
where
$$I^S(f) := \sum^\infty_{n=0} I^S_n(f_n)$$
for any
$$f=(f_n)_{n=0}^\infty \in \bigoplus^\infty_{n=0} L^2(\R^n_+) \cong \F.$$
\end{proposition}

This means that each element of $L^2(\S,\tau)$ has a unique representation in the form $I^S(f)$ for some $f\in \bigoplus^\infty_{n=0} L^2(\R^n_+)$, to which we refer as its \emph{chaos decomposition}.

There is a similar decomposition for $L^2$-biprocesses. Since the mapping $I^S: \F \to L^2(\S,\tau)$ gives rise to an isometric isomorphism
$$I^S \otimes I^S:\ \F \otimes \F \to L^2(\S,\tau) \otimes L^2(\S,\tau),$$
we see, by using the natural isometric identifications
$$L^2(\R_+, \F \otimes \F) \cong \F \otimes L^2(\R_+) \otimes \F$$
and
$$L^2(\R_+, L^2(\S,\tau) \otimes L^2(\S,\tau)) \cong L^2(\S,\tau) \otimes L^2(\R_+) \otimes L^2(\S,\tau) \cong \B_2,$$
that $I^S \otimes I^S$ induces an isometric isomorphism
$$I^S \otimes I^S:\ L^2(\R_+, \F\otimes\F) \to \B_2,$$
which is again denoted by $I^S \otimes I^S$. More explicitly, this induced isomorphism sends each $f: \R_+ \to \F\otimes\F, t\mapsto f_t$ that belongs to $L^2(\R_+,\F \otimes \F)$ to the $L^2$-biprocess that is given by $t \mapsto (I^S \otimes I^S)(f_t)$.

The following diagram offers a clear view on the situation described above.

$$\begin{xy}
\xymatrix{
\F \otimes L^2(\R_+) \otimes \F \ar[rr]^{\cong} \ar[dd]_{I^S \otimes \id \otimes I^S} & & L^2(\R_+, \F\otimes\F) \ar@{-->}[dd]^{I^S\otimes I^S}\\
                                                & &\\
L^2(\S,\tau) \otimes L^2(\R_+) \otimes L^2(\S,\tau) \ar[rr]^/1cm/{\cong} & & \B_2
}
\end{xy}$$

We call $U = (I^S \otimes I^S)(f)$ for $f\in L^2(\R_+, \F\otimes\F)$ the \emph{Wigner chaos expansion of the $L^2$-biprocess $U$}.

\subsection{Free Malliavin calculus}

Like in the classical case, the basic operators of free Malliavin calculus are constructed first on the side of the full Fock space and are then transferred to the algebra of field operators via the identification that is provided by the map $X \mapsto X \Omega$.

\subsubsection{Free Malliavin calculus on $\F(H)$}

As above in the construction of the free Brownian motion, we begin with the general case of an arbitrary complex Hilbert space $H$. On the full Fock space $\F(H)$ over $H$, we consider
\begin{itemize}
 \item an unbounded linear operator $$\tilde{\nabla}:\ \F(H) \supseteq D(\tilde{\nabla}) \to \F(H) \otimes H \otimes \F(H)$$ with domain $D(\tilde{\nabla}) = \F_\alg(H)$, which is determined by the conditions $\tilde{\nabla} \Omega = 0$ and $$\tilde{\nabla}(h_1 \otimes \dots \otimes h_n) := \sum^n_{j=1} (h_1 \otimes \dots \otimes h_{j-1}) \otimes h_j \otimes (h_{j+1} \otimes \dots \otimes h_n),$$ where the tensor products appearing in the brackets are understood as $\Omega$ if the corresponding set of indices happens to be empty.
 \item an unbounded linear operator $$\tilde{\delta}:\ \F(H) \otimes H \otimes \F(H) \supseteq D(\delta) \to \F(H)$$ with domain $D(\tilde{\delta}) = \F_\alg(H) \odot H \odot \F_\alg(H)$ by linear extension of
\begin{eqnarray*}
\tilde{\delta}((h_1 \otimes \dots \otimes h_n) \otimes h \otimes (g_1 \otimes \dots \otimes g_m)) &:=& h_1 \otimes \dots \otimes h_n \otimes h \otimes g_1 \otimes \dots \otimes g_m,\\
\tilde{\delta}(\Omega \otimes h \otimes (g_1 \otimes \dots \otimes g_m)) &:=& h \otimes g_1 \otimes \dots \otimes g_m,\\
\tilde{\delta}((h_1 \otimes \dots \otimes h_n) \otimes h \otimes \Omega) &:=& h_1 \otimes \dots \otimes h_n \otimes h,\\
\tilde{\delta}(\Omega \otimes h \otimes \Omega) &:=& h
\end{eqnarray*}
 \item an unbounded linear operator $$\tilde{N}:\ \F(\H) \supseteq D(N) \to \F(\H)$$ with domain $D(\tilde{N}) = \F_\alg(H)$, which is defined by $\tilde{N} \Omega = 0$ and $$\tilde{N}(h_1 \otimes \cdots \otimes h_n) := n\, h_1 \otimes \cdots \otimes h_n.$$ 
\end{itemize}

We collect now a few observations related to the operators $\tilde{\nabla}$ and $\tilde{\delta}$. We grant that some of these statements might appear quite artificial at the first sight, but there actual meaning will become clear after passing from the Fock space to operators defined on it.

\begin{remark}\label{rem:Fock-operators}
Consider the setting that was described above.
\begin{itemize}
 \item[(a)] A straightforward calculation shows that
\begin{equation}\label{eq:Fock-adjoints}
\langle\tilde{\nabla} y, u\rangle_{\F(H) \otimes H \otimes \F(H)} = \langle y, \tilde{\delta}(u)\rangle_{\F(H)}
\end{equation}
holds for all $y \in D(\tilde{\nabla})$ and $u \in D(\tilde{\delta})$.
 \item[(b)] If we endow $\F_\alg(H)$ with the multiplication induced by the tensor product $\otimes$ (in fact, we obtain in this way the tensor algebra over $H$), we may easily check that $\tilde{\nabla}$ satisfies a kind of product rule, namely
\begin{equation}\label{eq:Fock-derivation}
\tilde{\nabla}(y_1 \otimes y_2) = (\tilde{\nabla} y_1) \cdot y_2 + y_1 \cdot (\tilde{\nabla} y_2)
\end{equation}
for all $y_1,y_2 \in D(\tilde{\nabla})$, where $\cdot$ denotes the canonical left and right action, respectively, of $\F_\alg(H)$ on $\F_\alg(H) \otimes H \otimes \F_\alg(H)$ that is induced by $\otimes$, i.e. $$y_1 \cdot (x_1 \otimes h \otimes x_2) \cdot y_2 = (y_1 \otimes x_1) \otimes h \otimes (x_2 \otimes y_2).$$
 \item[(c)] Since the range of $\tilde{\nabla}$ is by definition contained in the domain of $\tilde{\delta}$, the composition $\tilde{\delta} \circ \tilde{\nabla}$ is well-defined. In fact, one has $\tilde{N} = \tilde{\delta} \circ \tilde{\nabla}$.
\end{itemize}
\end{remark}

\subsubsection{Free Malliavin calculus on $\F(L^2(\R_+))$}

We apply now the preceding construction in the special case, where the Hilbert space $H$ is given by $L^2(\R_+)$. Thus, we may use the isomorphisms
$$I^S:\ \F \to L^2(\S,\tau) \qquad\text{and}\qquad I^S \otimes I^S:\ L^2(\R_+, \F\otimes\F) \to \B_2$$
to pull over
\begin{itemize}
 \item the operator
$$\tilde{\nabla}:\ \F \supseteq D(\tilde{\nabla}) \to L^2(\R_+, \F\otimes\F)$$
to the so-called \emph{gradient operator}
$$\nabla:\ L^2(\S,\tau) \supseteq D(\nabla) \to \B_2$$
with domain $D(\nabla) = I^S(D(\tilde{\nabla}))$, 
 \item and the operator
$$\tilde{\delta}:\ L^2(\R_+, \F\otimes\F) \supseteq D(\tilde{\delta}) \to \F$$
to the so-called \emph{divergence operator}
$$\delta:\ \B_2 \supseteq D(\delta) \to L^2(\S,\tau)$$
with domain $D(\delta) = (I^S\otimes I^S)(D(\tilde{\delta}))$,
\end{itemize}
in the obvious way as shown in the following two commutative diagrams.
$$\begin{xy}
\xymatrix{
\F \ar[rr]^{I^S}                                                       & & L^2(\S,\tau)\\
D(\tilde{\nabla}) \ar@^{(->}[u] \ar[dd]_{\tilde{\nabla}} \ar[rr]^{I^S} & & D(\nabla) \ar@^{(->}[u] \ar@{-->}[dd]^\nabla\\
                                                                       & &\\
L^2(\R_+, \F\otimes\F) \ar[rr]^/0.5cm/{I^S \otimes I^S}                & & \B_2
}
\end{xy} \hspace{2cm}
\begin{xy}
\xymatrix{
L^2(\R_+,\F \otimes \F) \ar[rr]^/0.5cm/{I^S \otimes I^S}                           & & \B_2\\
D(\tilde{\delta}) \ar@^{(->}[u] \ar[dd]_{\tilde{\delta}} \ar[rr]^{I^S \otimes I^S} & & D(\delta) \ar@^{(->}[u] \ar@{-->}[dd]^\delta\\
                                                                                   & &\\
\F \ar[rr]^{I^S}                                                                   & & L^2(\S,\tau)
}
\end{xy}$$
In fact, the above definitions amount to
$$D(\nabla) = \S_\alg \qquad\text{and}\qquad D(\delta) = \S_\alg \odot L^2(\R) \odot \S_\alg,$$
where we abbreviate $\S_\alg := \S_\alg(L^2(\R_+,\R))$.

\begin{remark}
We may observe that the properties of the operators $\tilde{\nabla}$ and $\tilde{\delta}$, which were formulated in (a) and (b) of Remark \ref{rem:Fock-operators} take now a much more natural form. Indeed,
\begin{itemize}
 \item formula \eqref{eq:Fock-adjoints} reduces to
\begin{equation}\label{eq:adjoints}
\langle \nabla Y, U\rangle_{\B_2} = \langle Y, \delta(U)\rangle_{L^2(\S,\tau)}
\end{equation}
for all $Y\in D(\nabla)$ and $U\in D(\delta)$,
 \item and \eqref{eq:Fock-derivation} implies that $\nabla$ is a derivation in the sense that a kind of Leibniz rule
\begin{equation}\label{eq:Leibniz-rule}
\nabla(Y_1 Y_2) = (\nabla Y_1) \cdot Y_2 + Y_1 \cdot (\nabla Y_2)
\end{equation}
holds for all $Y_1,Y_2\in D(\nabla)$, where $\cdot$ denotes the left and right action, respectively, of $\S$ on $\B_2$.
\end{itemize}
\end{remark}

We recall \cite[Proposition 3.23]{KempNourdinPeccatiSpeicher2012}, which is itself a combination of Propositions 5.3.9 and 5.3.10 in \cite{BianeSpeicher1998}.

\begin{proposition}\label{prop:gradient}
The gradient operator
$$\nabla:\ L^2(\S,\tau) \supseteq D(\nabla) \to \B_2$$
is densely defined and closable. The domain $D(\overline{\nabla})$ of the closure
$$\overline{\nabla}:\ L^2(\S,\tau) \supseteq D(\overline{\nabla}) \to \B_2$$
can be characterized by the chaos expansion in the following way
$$D(\overline{\nabla}) = \Big\{I^S(f) \Big|\ f=(f_n)_{n=0}^\infty \in \F:\ \sum^\infty_{n=0} n \|f_n\|^2_{L^2(\R_+^n)} < \infty\Big\}.$$
In fact, if we write $Y\in D(\overline{\nabla})$ in the form $Y = I^S(f)$ with $f\in\F$, we have that
$$\|\overline{\nabla} Y\|_{\B_2}^2 = \sum^\infty_{n=0} n \|f_n\|^2_{L^2(\R_+^n)}.$$
Moreover, the action of $\overline{\nabla}$ on its domain $D(\overline{\nabla})$ is determined by
\begin{align*}
\lefteqn{\overline{\nabla}_t \Big(\int f(t_1,\dots,t_n)\, dS_{t_1} \cdots dS_{t_n}\Big)}\\
&= \sum^n_{j=1} \int f(t_1,\dots,t_{j-1},t,t_{j+1},\dots,t_n)\, dS_{t_1} \cdots dS_{t_{j-1}} \otimes dS_{t_{j+1}} \cdots dS_{t_n}
\end{align*}
for $f\in L^2(\R_+^n)$.
\end{proposition}

\begin{remark}\label{rem:domain-gradient}
We point out that Proposition 5.2.3 in \cite{BianeSpeicher1998} shows beyond this that $\nabla$ is also closable as an unbounded linear operator from $L^p(\S,\tau)$ to $\B_p$ for each $1 \leq p < \infty$. The domain of its closure, which will be denoted by $\D^p$, is given as the closure of $\S_\alg$ with respect to the norm $\|\cdot\|_{1,p}$ defined by
$$\|Y\|_{1,p} := \big(\|Y\|^p_{L^p(\S,\tau)} + \|\nabla Y\|_{\B_p}^p\big)^\frac{1}{p}.$$
We will use this observation only in the case $p=2$, where $D(\overline{\nabla}) = \D^2$ gives an alternative description of the domain $D(\overline{\nabla})$ of the closure of the gradient operator $\nabla$, which was characterized in Proposition \ref{prop:gradient} in terms of the chaos decomposition.
\end{remark}

Concerning now the divergence operator, we record here \cite[Proposition 3.25]{KempNourdinPeccatiSpeicher2012}, which combines Propositions 5.3.9 and 5.3.11 of \cite{BianeSpeicher1998}.

\begin{proposition}\label{prop:divergence}
The divergence operator
$$\delta:\ \B_2 \supseteq D(\delta) \to L^2(\S,\tau)$$
is densely defined and closable. The domain $D(\overline{\delta})$ of its closure
$$\overline{\delta}:\ \B_2 \supseteq D(\overline{\delta}) \to L^2(\S,\tau)$$
contains all adapted $L^2$-biprocesses $\B_2^a$ and for each $U\in\B_2^a$, we have
$$\overline{\delta}(U) = \int_{\R_+} U_t \sharp dS_t.$$
In general, the action of $\overline{\delta}$ on its domain $D(\overline{\delta})$ is determined by
\begin{align*}
\lefteqn{\overline{\delta} \Big(\int f_t(t_1,\dots,t_n; s_1,\dots,s_m)\, dS_{t_1} \cdots dS_{t_n} \otimes dS_{s_1} \cdots dS_{s_m}\Big)}\\
&= \int f_t(t_1,\dots,t_n; s_1,\dots,s_m)\, dS_{t_1} \cdots dS_{t_n} dS_t dS_{s_1} \cdots dS_{s_m}
\end{align*}
for any $f\in L^2(\R_+,L^2(\R^n_+) \otimes L^2(\R^m_+))$.
\end{proposition}

Finally, we also take the operator $\tilde{N}$ into account. This operator induces the so-called \emph{number operator}
$$N:\ L^2(\S,\tau) \supseteq D(N) \to L^2(\S,\tau)$$
with domain $D(N) := I^S(\F_\alg) = \S_\alg$ as shown in the following commutative diagram.
$$\begin{xy}
\xymatrix{
\F \ar[rr]^{I^S}                                             & & L^2(\S,\tau)\\
D(\tilde{N}) \ar@^{(->}[u] \ar[dd]_{\tilde{N}} \ar[rr]^{I^S} & & D(N) \ar@^{(->}[u] \ar@{-->}[dd]^N\\
                                                             & &\\
\F \ar[rr]^{I^S}                                             & & L^2(\S,\tau)
}
\end{xy}$$

\begin{remark}
The relation $\tilde{N} = \tilde{\delta} \circ \tilde{\nabla}$ on $\F_\alg$, which was recorded in part (c) of Remark \ref{rem:Fock-operators}, translates by definition immediately to the relation $N = \delta \circ \nabla$ on $\S_\alg$.
\end{remark}

We recall now \cite[Remark 3.24]{KempNourdinPeccatiSpeicher2012}.

\begin{proposition}\label{prop:number-operator}
The number operator
$$N:\ L^2(\S,\tau) \supseteq D(N) \to L^2(\S,\tau)$$
is densely defined and closable. The domain $D(\overline{N})$ of its closure can be characterized by using the chaos expansion in the following way
$$D(\overline{N}) = \Big\{I^S(f) \Big|\ f=(f_n)_{n=0}^\infty \in \F:\ \sum^\infty_{n=0} n^2 \|f_n\|^2_{L^2(\R_+^n)} < \infty\Big\}.$$
In particular, the closure of the gradient $\overline{\nabla}$ maps $D(\overline{N})$ into $D(\overline{\delta})$, and on $D(\overline{N})$, it holds true that $\overline{N} = \overline{\delta} \circ \overline{\nabla}|_{D(\overline{N})}$.  
\end{proposition}

\section{Non-commutative derivations}
\label{sec:Derivations}

The second pillar of the proof of Theorem \ref{MainThm} is the theory of non-commutative derivations as it arises from the work of Voiculescu \cite{Voi-Entropy-V, Voi-Entropy-VI} and of Dabrowski \cite{Dab-Gamma, Dab-free_stochastic_PDE}, and the corresponding generalization of methods originating from \cite{MaiSpeicherWeber2015}.

Derivations are mainly characterized by the Leibniz rule, which is a straightforward generalization of the Leibniz rule for usual derivatives. Hence, these objects can be introduced and studied in a purely algebraic setting: if $\A$ is any unital complex algebra and if $\M$ is an arbitrary $\A$-bimodule, we call a linear mapping $\delta:\ \A \to \M$ a \emph{$\M$-valued derivation on $\A$}, if it satisfies the Leibniz rule
$$\delta(x_1 x_2) = \delta(x_1) \cdot x_2 + x_1 \cdot \delta(x_2) \qquad\text{for all $x_1,x_2\in\A$}.$$

But since we are interested more in the analytic rather than the purely algebraic properties of derivations, we will impose here some additional conditions on the algebra $\A$ and the $\A$-bimodule $\M$.

For doing this, we clearly have a lot of flexibility. The most general notion of such analytic derivations is probably the one that is presented in \cite[Definition 4.1]{CiprianiSauvageot2003}. However, the feasibility of our arguments here depends strongly on more restrictive assumptions, due to which those derivations will behave pretty much like the usual non-commutative derivatives, as they appear for instance in \cite{Voi-Entropy-V}. Accordingly, we shall call them non-commutative derivations.

Throughout this section, let $(M,\tau)$ be a tracial $W^\ast$-probability space. 

\begin{definition}\label{def:noncommutative-derivation}
A linear map
$$\delta:\ M\supseteq D(\delta) \to L^2(M,\tau) \otimes L^2(M,\tau)$$
is called a \emph{non-commutative derivation on $M$} if the following two conditions are satisfied:
\begin{itemize}
 \item The domain $D(\delta)$ of $\delta$ is a unital $\ast$-subalgebra of $M$, which is moreover weakly dense in $M$.
 \item The linear map $\delta$ satisfies the \emph{Leibniz rule} (or \emph{product rule})
$$\delta(x_1 x_2) = \delta(x_1) \cdot x_2 + x_1 \cdot \delta(x_2)$$
for all $x_1,x_2\in D(\delta)$, where $\cdot$ denotes the natural bimodule operation of $M$ on the Hilbert space $$L^2(M,\tau) \otimes L^2(M,\tau) \cong L^2(M \otimes M, \tau\otimes\tau).$$
\end{itemize}
\end{definition}

\begin{remark}\label{rem:noncommutative-derivatives}
Let $\C\langle x_1,\dots,x_n\rangle$ denote the $\ast$-algebra of non-commutative polynomials in formal self-adjoint variables $x_1,\dots,x_n$.

If $P$ is any monomial in the variables $x_1,\dots,x_n$, we put for any fixed $i=1,\dots,n$
$$\partial_i P := \sum_{P = P_1 x_i P_2} P_1 \otimes P_2$$
where the sum runs over all decompositions of $P$ in the form $P = P_1 x_i P_2$ with some monomials $P_1,P_2$. In particular, we have $\partial_i x_j = \delta_{i,j} 1\otimes 1$ for $i,j=1,\dots,n$. It is easy to check that $\partial_i$ extends by linearity to a $\C\langle x_1,\dots,x_n\rangle^{\odot 2}$-valued derivation on $\C\langle x_1,\dots,x_n\rangle$.

Conversely, as a linear map
$$\partial_i:\ \C\langle x_1,\dots,x_n\rangle \to \C\langle x_1,\dots,x_n\rangle \odot \C\langle x_1,\dots,x_n\rangle,$$
the \emph{non-commutative derivative $\partial_i$} is uniquely determined by the Leibniz rule and the property $\partial_i x_j = \delta_{i,j} 1\otimes 1$ for $i,j=1,\dots,n$.

In the case $n=1$, we will abbreviate $\partial_1$ simply by $\partial$.

Assume now that $\delta: M\supseteq D(\delta) \to L^2(M,\tau) \otimes L^2(M,\tau)$ is any non-commutative derivation in the sense of Definition \ref{def:noncommutative-derivation}. If $X_1,\dots,X_n$ are self-adjoint elements in $D(\delta)$ and if $P\in \C\langle x_1,\dots,x_n\rangle$ is any non-commutative polynomial, then the evaluation $P(X_1,\dots,X_n)$ belongs clearly to $D(\delta)$ and we have the formula
\begin{equation}\label{eq:noncommutative-derivatives}
\delta(P(X_1,\dots,X_n)) = \sum^n_{i=1} (\partial_i P)(X_1,\dots,X_n) \sharp \delta(X_i),
\end{equation}
where we abbreviate by $Q(X)$ the natural evaluation of any $Q\in \C\langle x_1,\dots,x_n\rangle^{\odot 2}$ at $(X_1,\dots,X_n)$. In other words, the non-commutative derivatives $\partial_1,\dots,\partial_n$ are universal in the sense that they provide an explicit expression for the restriction of any non-commutative derivation $\delta$ to a subalgebra $\C\langle X_1,\dots,X_n\rangle$ of its domain $D(\delta)$ in terms of its values on the generators $X_1,\dots,X_n$.
\end{remark}

Following \cite{Voi-Entropy-V,Voi-Entropy-VI}, we change now our point of view by considering any non-commutative derivation $\delta: M \supseteq D(\delta) \to L^2(M,\tau) \otimes L^2(M,\tau)$ in the sense of Definition \ref{def:noncommutative-derivation} as an unbounded linear operator
$$\delta:\ L^2(M,\tau) \supset D(\delta) \to L^2(M,\tau) \otimes L^2(M,\tau).$$
Since $D(\delta)$ is clearly dense in $L^2(M,\tau)$ with respect to the $L^2$-norm $\|\cdot\|_2$ induced by $\tau$, we can also consider its adjoint operator
$$\delta^\ast:\ L^2(M,\tau) \otimes L^2(M,\tau) \supseteq D(\delta^\ast) \to L^2(M,\tau).$$

The theory that we are going to presented in the next subsections concerns properties of $\delta$ and its adjoint $\delta^\ast$. More precisely, we will discuss the question of closability for $\delta$ and we will show that $\delta$ and $\delta^\ast$, which are unbounded operators by definition, can nevertheless be controlled in appropriate norms. For most of these results, the condition $1 \otimes 1 \in D(\delta^\ast)$ turns out to be crucial.

\subsection{Voiculescu's formulas for $\delta^\ast$}

In \cite{Voi-Entropy-V}, Voiculescu deduced formulas for the adjoint operator $\delta^\ast$ of a non-commutative derivation $\delta$ under the assumption that $1\otimes 1 \in D(\delta^\ast)$. This was shown in \cite{Voi-Entropy-V} only in the case of the non-commutative derivatives that are defined on the algebra of finitely many generators, but it was noted and worked out in \cite{Voi-Entropy-VI} that the same arguments apply in more general situations. Although this is commonly accepted as a well-known fact, we give here for reader's convenience a complete introduction to this circle of ideas, since these beautiful results are of great importance for our considerations.

For the rest of this subsection, let $\delta: M \supseteq D(\delta) \to L^2(M,\tau) \otimes L^2(M,\tau)$ be some fixed non-commutative derivation in the sense of Definition \ref{def:noncommutative-derivation}, viewed as an unbounded linear operator
$$\delta:\ L^2(M,\tau) \supset D(\delta) \to L^2(M,\tau) \otimes L^2(M,\tau).$$
Following Voiculescu's strategy, we begin by deducing some very useful product rules for its adjoint operator $\delta^\ast$.

Clearly, we may extend the involution $\ast$ on $M$ from $M$ uniquely to an involution on $L^2(M,\tau)$, and the canonical involution $\ast$ on $M \otimes M$ from $M\otimes M$ uniquely to an involution $L^2(M,\tau) \otimes L^2(M,\tau)$. Consequently, 
\begin{equation}\label{eq:inner-product-involution-1}
\langle x,y\rangle = \langle y^\ast, x^\ast\rangle
\end{equation}
holds for all $x,y \in L^2(M,\tau)$.

\begin{lemma}\label{lem:Voi-lemma-0}
Let $u \in D(\delta^\ast) \cap (M \odot M)$ and $x\in D(\delta)$ be given. Then
\begin{equation}\label{eq:Voi-formula-0}
\begin{aligned}
\delta^\ast(x \cdot u) &= x \delta^\ast(u) - (\tau\otimes\id)(u \sharp \delta(x^\ast)^\ast),\\
\delta^\ast(u \cdot x) &= \delta^\ast(u) x - (\id\otimes\tau)(u \sharp \delta(x^\ast)^\ast),
\end{aligned}
\end{equation}
where $\sharp$ is defined according to Remark \ref{rem:sharp} with respect to the $M$-$M$-bimodule $L^2(M,\tau) \otimes L^2(M,\tau)$. In particular, for any $u\in D(\delta^\ast) \cap (M \odot M)$, we have
$$\{x_1 \cdot u \cdot x_2|\ x_1,x_2\in D(\delta)\} \subseteq D(\delta^\ast).$$
\end{lemma}

\begin{proof}
Let $u\in D(\delta^\ast) \cap (M \odot M)$ and $x\in D(\delta)$ be given. For any $y\in D(\delta)$, we observe that
\begin{align*}
\langle \delta(y), x\cdot u\rangle
&= \langle x^\ast \cdot \delta(y), u \rangle\\
&= \langle \delta(x^\ast y), u \rangle - \langle \delta(x^\ast) \cdot y, u \rangle\\
&= \langle x^\ast y, \delta^\ast(u) \rangle - \langle 1 \otimes y, u\sharp \delta(x^\ast)^\ast\rangle\\
&= \langle y, x \delta^\ast(u) \rangle - \langle y, (\tau\otimes\id)(u\sharp \delta(x^\ast)^\ast)\rangle\\
&= \langle y, x \delta^\ast(u) - (\tau\otimes\id)(u\sharp \delta(x^\ast)^\ast)\rangle,
\end{align*}
from which $x\cdot u\in D(\delta^\ast)$ and the first formula in \eqref{eq:Voi-formula-0} follows. Analogously, we obtain by
\begin{align*}
\langle \delta(y), u\cdot x\rangle
&= \langle \delta(y) \cdot x^\ast, u \rangle\\
&= \langle \delta(y x^\ast), u \rangle - \langle y \cdot \delta(x^\ast), u \rangle\\
&= \langle y x^\ast, \delta^\ast(u) \rangle - \langle y \otimes 1, u\sharp \delta(x^\ast)^\ast\rangle\\
&= \langle y, \delta^\ast(u) x \rangle - \langle y, (\id\otimes\tau)(u\sharp \delta(x^\ast)^\ast)\rangle\\
&= \langle y, \delta^\ast(u) x - (\id\otimes\tau)(u\sharp \delta(x^\ast)^\ast)\rangle
\end{align*}
that $u\cdot x \in D(\delta^\ast)$ and the second formula in \eqref{eq:Voi-formula-0}. A combination of both observations immediately yields the stated inclusion
$$\{x_1 \cdot u \cdot x_2|\ x_1,x_2\in D(\delta)\} \subseteq D(\delta^\ast)$$
for any $u\in D(\delta^\ast) \cap (M \odot M)$.
\end{proof}

In the case $1 \otimes 1 \in D(\delta^\ast)$, Lemma \ref{lem:Voi-lemma-0} yields an explicit formula for $\delta^\ast$ on $D(\delta) \odot D(\delta)$ in terms of $\delta^\ast(1\otimes 1)$ and $\delta$. It takes its nicest form if we require an additional property of $\delta$. In fact, we will assume a certain compatibility between the involution $\ast$ on $M$ and the involution $\dagger$ on $M \otimes M$, where the latter is determined by
$$(x_1 \otimes x_2)^\dagger := x_2^\ast \otimes x_1^\ast.$$
Note that $\dagger$ differs from the canonical involution $\ast$ on $M\otimes M$ only by the flip mapping $\sigma: M\otimes M \to M\otimes M$, i.e., we have $u^\dagger = \sigma(u^\ast)$.

Clearly, we may extend the involution $\dagger$ from $M\otimes M$ uniquely to an involution $L^2(M,\tau) \otimes L^2(M,\tau)$. Accordingly, for all $u,v\in L^2(M,\tau) \otimes L^2(M,\tau)$, it holds true that
\begin{equation}\label{eq:inner-product-involution-2}
\langle u,v\rangle = \langle v^\dagger, u^\dagger\rangle.
\end{equation}

\begin{definition}\label{def:noncommutative-derivation-real}
A non-commutative derivation
$$\delta: M \supseteq D(\delta) \to L^2(M,\tau) \otimes L^2(M,\tau)$$
on $(M,\tau)$ is called \emph{real}, if it satisfies
\begin{equation}\label{eq:noncommutative-derivation-real}
\delta(x)^\dagger = \delta(x^\ast) \qquad\text{for all $x\in D(\delta)$}.
\end{equation}
\end{definition}

Often, condition \eqref{eq:noncommutative-derivation-real} can be weakened. We record this here as a remark.

\begin{remark}\label{rem:real-generators}
We point out that condition \eqref{eq:noncommutative-derivation-real} is automatically satisfied if the unital $\ast$-algebra $D(\delta)$ is generated by self-adjoint elements $x_i$, $i\in I$, for some index set $I\neq\emptyset$, such that $\delta(x_i)^\dagger = \delta(x_i)$ holds for all $i\in I$.

Indeed, if we define $\tilde{\delta}$ with $D(\tilde{\delta}) := D(\delta)$ by
$$\tilde{\delta}:\ D(\tilde{\delta}) \to L^2(M,\tau) \otimes L^2(M,\tau),\ x \mapsto \delta(x^\ast)^\dagger,$$
we can easily check that $\tilde{\delta}$ is a non-commutative derivation as well. Thus, the set
$$D := \{x\in D(\delta)|\ \delta(x) = \tilde{\delta}(x)\}$$
is closed under multiplication, i.e. $x_1,x_2\in D$ implies $x_1 x_2 \in D$. Since it contains the generators $\{x_i|\ i\in I\}$ by assumption, we must have that $D=D(\delta)$, from which it follows by construction that $\delta(x)^\dagger = \delta(x^\ast)$ holds for all $x\in D(\delta)$.
\end{remark}

The following lemma collects some useful formulas for real non-commutative derivations.

\begin{lemma}\label{lem:real-derivations-formulas}
Let $\delta: M \supseteq D(\delta) \to L^2(M,\tau) \otimes L^2(M,\tau)$ be a real non-commutative derivation on $(M,\tau)$, Then, for all $x\in D(\delta)$, it holds true that
\begin{align*}
(\id\otimes\tau)(\delta(x))^\ast &= (\tau\otimes\id)(\delta(x^\ast)),\\
(\tau\otimes\id)(\delta(x))^\ast &= (\id\otimes\tau)(\delta(x^\ast)).
\end{align*}
Furthermore, for any $u\in D(\delta^\ast)$, we have also $u^\dagger \in D(\delta^\ast)$ and it holds true that
$$\delta^\ast(u^\dagger) = \delta^\ast(u)^\ast.$$
In particular, if $1 \otimes 1 \in D(\delta^\ast)$, we have $\delta^\ast(1 \otimes 1) = \delta^\ast(1 \otimes 1)^\ast$.
\end{lemma}

\begin{proof}
The first statement is an immediate consequence of the defining property of real derivations, since in general
\begin{equation}\label{eq:involutions-flip}
\begin{aligned}
(\id\otimes\tau)(u)^\ast &= (\tau\otimes\id)(u^\dagger),\\
(\tau\otimes\id)(u)^\ast &= (\id\otimes\tau)(u^\dagger)
\end{aligned}
\end{equation}
holds for each $u\in L^2(M,\tau) \otimes L^2(M,\tau)$. For seeing the second statement, we take any $y\in D(\delta)$ and we observe by using \eqref{eq:inner-product-involution-2} that
$$\langle u^\dagger, \delta(y)\rangle = \langle \delta(y)^\dagger, u\rangle = \langle \delta(y^\ast), u\rangle = \langle y^\ast, \delta^\ast(u)\rangle = \langle \delta^\ast(u)^\ast, y\rangle.$$
This yields $u^\dagger \in D(\delta^\ast)$ with $\delta^\ast(u^\dagger) = \delta^\ast(u)^\ast$, as desired.
\end{proof}

Now, we can combine formulas \eqref{eq:Voi-formula-0} of Lemma \ref{lem:Voi-lemma-0}.

\begin{lemma}\label{lem:Voi-lemma}
If the condition $1 \otimes 1 \in D(\delta^\ast)$ is satisfied, then
$$D(\delta) \odot D(\delta) \subseteq D(\delta^\ast).$$
If $\delta$ is a real derivation in the sense of Definition \ref{def:noncommutative-derivation-real}, then we have more explicitly for all $u \in D(\delta) \odot D(\delta)$ that
\begin{equation}\label{eq:Voi-formula}
\delta^\ast(u) = u \sharp \delta^\ast(1\otimes 1) - m_1 (\id\otimes\tau\otimes\id)(\delta \otimes \id + \id \otimes \delta)(u),
\end{equation}
where, in general, we denote by $m_\eta$ for any $\eta\in L^2(M,\tau)$ the linear mapping $m_\eta: M \odot M \to L^2(M,\tau)$ that is determined by $m_\eta(v) = v \sharp \eta$, so that $m_1$ is nothing else than the multiplication map $m_1(x_1 \otimes x_2) = x_1 x_2$.
\end{lemma}

The formula \eqref{eq:Voi-formula} given in Lemma \ref{lem:Voi-lemma} immediately implies that in particular
\begin{equation}\label{eq:Voi-formula-red}
\begin{aligned}
\delta^\ast(x \otimes 1) &= x \delta^\ast(1 \otimes 1) - (\id\otimes\tau)(\delta(x)),\\
\delta^\ast(1 \otimes x) &= \delta^\ast(1 \otimes 1) x - (\tau\otimes\id)(\delta(x)),
\end{aligned}
\end{equation}
which we record here for later reference.

\begin{proof}[Proof of Lemma \ref{lem:Voi-lemma}]
The first assertion, namely that $D(\delta) \odot D(\delta) \subseteq D(\delta^\ast)$ holds under the condition $1 \otimes 1 \in D(\delta^\ast)$, is an immediate consequence of Lemma \ref{lem:Voi-lemma-0}. Note that we did not use for this conclusion the assumption that $\delta$ is real.

For seeing \eqref{eq:Voi-formula}, we proceed as follows. First of all, we note that the validity of \eqref{eq:noncommutative-derivation-real} guarantees according to Lemma \ref{lem:real-derivations-formulas} that
\begin{align*}
(\id\otimes\tau)(\delta(x^\ast)^\ast) &= (\tau\otimes\id)(\delta(x)),\\
(\tau\otimes\id)(\delta(x^\ast)^\ast) &= (\id\otimes\tau)(\delta(x))
\end{align*}
for each $x\in D(\delta)$. Next, for any $u = x_1 \otimes x_2$ with $x_1,x_2\in D(\delta)$, we check by using consecutively both formulas of \eqref{eq:Voi-formula-0} and Lemma \ref{lem:real-derivations-formulas} that
\begin{align*}
\delta^\ast(u)
&= \delta^\ast(x_1 \cdot (1 \otimes x_2))\\
&= x_1 \delta^\ast((1 \otimes 1) \cdot x_2) - (\tau\otimes\id)((1\otimes x_2) \sharp \delta(x_1^\ast)^\ast)\\
&= x_1 \delta^\ast(1 \otimes 1) x_2 - x_1 (\id\otimes\tau)(\delta(x_2^\ast)^\ast)- (\tau\otimes\id)((1\otimes x_2) \sharp \delta(x_1^\ast)^\ast)\\
&= u\sharp \delta^\ast(1 \otimes 1) - x_1 (\tau\otimes\id)(\delta(x_2)) - (\id\otimes\tau)(\delta(x_1)) x_2\\
&= u \sharp \delta^\ast(1\otimes 1) - m_1 (\id\otimes\tau\otimes\id)(\delta \otimes \id + \id \otimes \delta)(u).
\end{align*}
By linearity, this shows \eqref{eq:Voi-formula} for all $u\in D(\delta) \odot D(\delta)$. This concludes the proof.
\end{proof}

\subsection{Dabrowski's inequalities}

Based on Voiculescu's formulas, Dabrowski deduced in \cite{Dab-Gamma} a collection of interesting inequalities concerning the boundedness of the non-commutative derivatives, which are very surprising from a classical point of view. In \cite{Dab-free_stochastic_PDE}, he noted that the same arguments also apply in a more general setting. More precisely, he observed (without carrying out the proof) that his result remain valid for any real derivation, which satisfies in addition the so-called coassociativity relation.

\begin{definition}\label{def:noncommutative-derivation-coassociative}
Let $\delta: M \supseteq D(\delta) \to L^2(M,\tau) \otimes L^2(M,\tau)$ be a non-commutative derivation on $(M,\tau)$. We say that $\delta$ satisfies the \emph{coassociativity relation},
\begin{itemize}
 \item if $\delta$ takes its values in $D(\delta) \odot D(\delta)$,
 \item and if $\delta$ has the property that
       \begin{equation}\label{eq:coassociativity}
       (\delta \otimes \id) \circ \delta = (\id \otimes \delta) \circ \delta.
       \end{equation}
\end{itemize}
\end{definition}

For reader's convenience, we state here those of Dabrowski's formulas, which we need for our purposes. Since it is instructive, we also include a slightly simplified proof thereof.

\begin{theorem}\label{thm:Dab-lemma}
Let $\delta: M \supseteq D(\delta) \to L^2(M,\tau) \otimes L^2(M,\tau)$ be a non-commutative derivation on a tracial $W^\ast$-probability space $(M,\tau)$, which
\begin{itemize}
 \item is real in the sense of Definition \ref{def:noncommutative-derivation-real}
 \item and satisfies the coassociativity relation as formulated in Definition \ref{def:noncommutative-derivation-coassociative}.
\end{itemize}
If the condition $1 \otimes 1 \in D(\delta^\ast)$ is satisfied, we have for all $x\in D(\delta)$ that
\begin{equation}\label{eq:Dab-inequalities-1}
\begin{aligned}
\|\delta^\ast(x\otimes 1)\|_2 & \leq \|\delta^\ast(1\otimes1)\|_2 \|x\|\\
\|\delta^\ast(1\otimes x)\|_2 & \leq \|\delta^\ast(1\otimes1)\|_2 \|x\|
\end{aligned}
\end{equation}
and
\begin{equation}\label{eq:Dab-inequalities-2}
\begin{aligned}
\|(\id\otimes\tau)(\delta(x))\|_2 &\leq 2 \|\delta^\ast(1\otimes1)\|_2 \|x\|\\
\|(\tau\otimes\id)(\delta(x))\|_2 &\leq 2 \|\delta^\ast(1\otimes1)\|_2 \|x\|
\end{aligned}
\end{equation}
\end{theorem}

Before proceeding with to the proof of Theorem \ref{thm:Dab-lemma}, we record here the following formula for its later use therein.

\begin{lemma}\label{lem:coassociativity-red}
In the situation of Theorem \ref{thm:Dab-lemma}, let $x\in D(\delta)$ be given and put
$$y := (\id\otimes\tau)(\delta(x)).$$
Then $y\in D(\delta)$ holds and we have that
$$(\id\otimes\tau)(\delta(y)) = (\id\otimes \langle \cdot, \delta^\ast(1 \otimes 1)\rangle)(\delta(x)).$$
\end{lemma}

\begin{proof}
Since $\delta$ is assumed to satisfy the coassociativity relation, we know by Definition \ref{def:noncommutative-derivation-coassociative} that in particular $D(\delta) \odot D(\delta)$ holds, which gives $y\in D(\delta)$. Furthermore, according to the coassociativity relation formulated in \eqref{eq:coassociativity}, we see that
\begin{align*}
\delta(y) &= (\id\otimes\id\otimes\tau)\big((\delta\otimes\id)(\delta(x))\big)\\
          &= (\id\otimes\id\otimes\tau)\big((\id\otimes\delta)(\delta(x))\big)
\end{align*}
holds. Since we have on $D(\delta)$ the identity $(\tau\otimes\tau)\circ\delta = \langle \cdot, \delta^\ast(1 \otimes 1)\rangle$, we get
\begin{align*}
(\id\otimes\tau)(\delta(y)) &= (\id\otimes\tau\otimes\tau)\big((\id\otimes\delta)(\delta(x))\big)\\
                            &= \big(\id\otimes((\tau\otimes\tau)\circ\delta)\big)(\delta(x))\\
                            &= \big(\id\otimes \langle \cdot, \delta^\ast(1 \otimes 1)\rangle \big)(\delta(x)),
\end{align*}
which is the desired formula.
\end{proof}

Additionally, the proof of Theorem \ref{thm:Dab-lemma} will be based on the following observation.

\begin{lemma}\label{lem:iterative-norm-bound}
Let $(M,\tau)$ be a $W^\ast$-probability space and let $T: D(T) \rightarrow M$ be a linear operator on a unital $\ast$-subalgebra $D(T)$ of $M$. Assume that the following conditions are satisfied:
\begin{itemize}
 \item[(i)] There exists a constant $C>0$ such that
$$\|T(x)\|^2_2 \leq C \|T(x^\ast x)\|_2 \qquad\text{for all $x\in D(T)$}.$$
 \item[(ii)] For each $x\in D(T)$, we have that
$$\limsup_{m\to\infty} \|T(x^m)\|^\frac{1}{m}_2 \leq \|x\|.$$
\end{itemize}
Then $T$ satisfies $\|T(x)\|_2 \leq C \|x\|$ for all $x\in D(T)$.
\end{lemma}

\begin{proof}
Let $x\in D(T)$ be given. For each $n\in\N_0$, we define $z_n := (x^\ast x)^{2^n} \in D(T)$. By assumption (i), we see that
$$\|T(z_n)\|^2_2 \leq C \|T(z_{n+1})\|_2 \qquad\text{for all $n\in\N_0$},$$
which yields inductively
$$\|T(z_0)\|_2 \leq C^{\frac{1}{2} + \dots + \frac{1}{2^n}} \|T(z_n)\|^{\frac{1}{2^n}}_2 \qquad\text{for all $n\in\N_0$}.$$
Since
$$\limsup_{n\to\infty} \|T(z_n)\|^{\frac{1}{2^n}}_2 = \limsup_{n\to\infty} \|T((x^\ast x)^{2^n})\|^{\frac{1}{2^n}}_2 \leq \|x^\ast x\| = \|x\|^2$$
due to (ii), it follows that
$$\|T(z_0)\|_2 \leq C \|x\|^2.$$
By using (ii) once again, we obtain
$$\|T(x)\|^2_2 \leq C \|T(z_0)\|_2 \leq C^2 \|x\|^2$$
and hence $\|T(x)\|_2 \leq C \|x\|$, as stated.
\end{proof}

\begin{proof}[Proof of Theorem \ref{thm:Dab-lemma}]
First of all, we note that it suffices to prove \eqref{eq:Dab-inequalities-1}, since \eqref{eq:Dab-inequalities-2} follows from \eqref{eq:Dab-inequalities-1} and Voiculescu's formula \eqref{eq:Voi-formula-red} by an application of the triangle inequality.

For proving \eqref{eq:Dab-inequalities-1}, we want to use Lemma \ref{lem:iterative-norm-bound}. We consider the linear mapping $T: D(T) \to M$ on $D(T) := D(\delta)$ given by
$$T(x) := \delta^\ast(x \otimes 1) \qquad\text{for all $x\in D(\delta)$}.$$
Since Lemma \ref{lem:Voi-lemma} guarantees $D(\delta) \odot D(\delta) \subseteq D(\delta^\ast)$, the mapping $T$ is indeed well-defined.

Now, we just have to follow the receipt given in Lemma \ref{lem:iterative-norm-bound}.

(i) For any given $x\in D(\delta)$, we have to compare $\|T(x^\ast x)\|_2$ and $\|T(x)\|_2$. In fact, we will show that
\begin{equation}\label{eq:T-formula}
\|T(x)\|^2_2 = \langle T(x^\ast x), \delta^\ast(1\otimes 1)\rangle
\end{equation}
from which
$$\|T(x)\|^2_2 \leq \|\delta^\ast(1\otimes 1)\|_2 \|T(x^\ast x)\|_2$$
immediately follows by an application of the Cauchy-Schwarz inequality.

Formula \eqref{eq:T-formula} can be shown as follows. Let $x\in D(\delta)$ be given and put $y := (\id\otimes\tau)(\delta(x))$. Since
$$y^\ast = (\id\otimes\tau)(\delta(x))^\ast = (\tau\otimes\id)(\delta(x^\ast))$$
according to Lemma \ref{lem:real-derivations-formulas}, we may observe by using in turn Lemma \ref{lem:coassociativity-red} and Lemma \ref{lem:real-derivations-formulas} in the version \eqref{eq:Voi-formula-red} that
\begin{align*}
\|y\|_2^2
&= \langle y, (\id\otimes\tau)(\delta(x))\rangle\\
&= \langle y \otimes 1, \delta(x)\rangle\\
&= \langle \delta^\ast(y\otimes 1), x\rangle\\
&= \langle y \delta^\ast(1\otimes 1), x\rangle - \langle (\id\otimes\tau)(\delta(y)), x\rangle\\
&= \langle \delta^\ast(1\otimes 1) x^\ast, y^\ast\rangle - \langle (\id\otimes\langle \cdot, \delta^\ast(1 \otimes 1)\rangle)(\delta(x)), x\rangle\\
&= \langle 1\otimes \delta^\ast(1\otimes 1) x^\ast, \delta(x^\ast)\rangle - \langle \delta(x), x \otimes \delta^\ast(1 \otimes 1)\rangle.
\end{align*}
Because moreover 
\begin{align*}
\langle \delta(x), & x \otimes \delta^\ast(1 \otimes 1)\rangle\\
&= \langle x^\ast \cdot \delta(x), 1 \otimes \delta^\ast(1 \otimes 1)\rangle\\
&= \langle \delta(x^\ast x), 1 \otimes \delta^\ast(1 \otimes 1)\rangle - \langle \delta(x^\ast) \cdot x, 1 \otimes \delta^\ast(1 \otimes 1)\rangle\\
&= \langle \delta(x^\ast x), 1 \otimes \delta^\ast(1 \otimes 1)\rangle - \langle \delta(x^\ast), 1 \otimes \delta^\ast(1 \otimes 1)x^\ast\rangle,
\end{align*}
we may conclude
$$\|y\|_2^2 = 2\Re\big( \langle 1\otimes \delta^\ast(1\otimes 1) x^\ast, \delta(x^\ast)\rangle \big) - \langle \delta(x^\ast x), 1 \otimes \delta^\ast(1 \otimes 1)\rangle.$$

Furthermore, since $T(x) = x\delta^\ast(1\otimes1) - y$ due to \eqref{eq:Voi-formula-red}, we get that
\begin{align*}
\|T(x)\|^2_2 &= \langle x\delta^\ast(1\otimes 1) - y, x\delta^\ast(1\otimes 1) - y\rangle\\
             &= \|x\delta^\ast(1\otimes 1)\|_2^2 + \|y\|^2_2 - 2 \Re\big( \langle x\delta^\ast(1\otimes 1), y\rangle\big)\\
             &= \|x\delta^\ast(1\otimes 1)\|_2^2 + \|y\|^2_2 - 2\Re\big( \langle x\delta^\ast(1\otimes 1) \otimes 1, \delta(x)\rangle\big).
\end{align*}

We check now
\begin{align*}
\langle x\delta^\ast(1\otimes 1)& \otimes 1, \delta(x)\rangle\\
&= \langle \delta^\ast(1\otimes 1), (\id\otimes\tau)(x^\ast\cdot \delta(x))\rangle\\
&= \langle (\id\otimes\tau)(x^\ast \cdot \delta(x))^\ast, \delta^\ast(1\otimes1)\rangle & \text{(by \eqref{eq:inner-product-involution-1})}\\
&= \langle (\id\otimes\tau)(\delta(x))^\ast x, \delta^\ast(1\otimes1)\rangle\\
&= \langle (\tau\otimes\id)(\delta(x^\ast)) x, \delta^\ast(1\otimes1)\rangle & \text{(by Lemma \ref{lem:real-derivations-formulas})}\\
&= \langle (\tau\otimes\id)(\delta(x^\ast)), \delta^\ast(1\otimes1) x^\ast\rangle\\
&= \langle \delta(x^\ast), 1\otimes \delta^\ast(1\otimes1) x^\ast\rangle,
\end{align*}
so that
$$\Re\big(\langle x\delta^\ast(1\otimes 1) \otimes 1, \delta(x)\rangle\big) = \Re\big(\langle 1\otimes \delta^\ast(1\otimes1) x^\ast, \delta(x^\ast)\rangle\big).$$

A combination of our previous computations leads us to
\begin{equation}\label{eq:T-formula-1}
\|T(x)\|^2_2 = \|x\delta^\ast(1\otimes 1)\|_2^2 - \langle \delta(x^\ast x), 1 \otimes \delta^\ast(1 \otimes 1)\rangle
\end{equation}

Furthermore, due to \eqref{eq:Voi-formula-red}, we have
$$T(x^\ast x) = x^\ast x \delta^\ast(1 \otimes 1) - (\id\otimes\tau)(\delta(x^\ast x)),$$ 
and hence
\begin{equation}\label{eq:T-formula-2}
\langle T(x^\ast x), \delta^\ast(1\otimes 1)\rangle = \|x \delta^\ast(1\otimes 1)\|_2^2 - \langle \delta(x^\ast x), \delta^\ast(1\otimes 1) \otimes 1\rangle.
\end{equation}

Since \eqref{eq:T-formula-1} implies that $\langle \delta(x^\ast x), 1\otimes \delta^\ast(1\otimes 1)\rangle$ must be real, we get by using \eqref{eq:inner-product-involution-2}, Lemma \ref{lem:real-derivations-formulas}, and \eqref{eq:noncommutative-derivation-real} that
$$\langle \delta(x^\ast x), \delta^\ast(1\otimes 1) \otimes 1\rangle = \langle 1 \otimes \delta^\ast(1\otimes 1), \delta(x^\ast x)\rangle = \langle  \delta(x^\ast x), 1 \otimes \delta^\ast(1\otimes 1)\rangle.$$

Thus, comparing \eqref{eq:T-formula-1} and \eqref{eq:T-formula-2} gives
$$\|T(x)\|^2_2 = \langle T(x^\ast x), \delta^\ast(1\otimes 1)\rangle,$$
which is the stated formula \eqref{eq:T-formula}.

(ii) To begin with, we observe that for any polynomial $P$ and any $x\in D(\delta)$
\begin{equation}\label{eq:T-bound}
\|T(P(x))\|_2 \leq \|P(x)\| \|\delta^\ast(1\otimes 1)\|_2 + \|(\partial P)(x)\|_\pi \|\delta(x)\|_2,
\end{equation}
where $\|\cdot\|_\pi$ denotes the projective norm on $D(\delta) \odot D(\delta)$, which is given by
$$\|u\|_\pi := \inf\Big\{\sum^N_{j=1} \|a_j\| \|b_j\|\, \Big|\, N\in\N,\, a_1,\dots,a_N,b_1,\dots,b_N\in D(\delta):\ u = \sum^N_{j=1} a_j \otimes b_j\Big\}$$
for any $u\in D(\delta) \odot D(\delta)$.

Indeed, according to \eqref{eq:Voi-formula-red}, we have for each polynomial $P$ and $x\in D(\delta)$
\begin{align*}
T(P(x)) &= \delta^\ast(P(x) \otimes 1)\\
        &= P(x) \delta^\ast(1\otimes 1) - (\id\otimes\tau)(\delta(P(x)))\\
        &= P(x) \delta^\ast(1\otimes 1) - (\id\otimes\tau)((\partial P)(x) \sharp \delta(x)),
\end{align*}
where we used that $\delta(P(x)) = (\partial P)(x) \sharp \delta(x)$ according to formula \eqref{eq:noncommutative-derivatives}, which was given in Remark \ref{rem:noncommutative-derivatives}. This yields as desired
\begin{align*}
\|T(P(x))\|_2 &\leq \|P(x) \delta^\ast(1\otimes 1)\|_2 + \|(\id\otimes\tau)((\partial P)(x) \sharp \delta(x))\|_2\\
              &\leq \|P(x) \delta^\ast(1\otimes 1)\|_2 + \|(\partial P)(x) \sharp \delta(x)\|_2\\
              &\leq \|P(x)\| \|\delta^\ast(1\otimes 1)\|_2 + \|(\partial P)(x)\|_\pi \|\delta(x)\|_2.
\end{align*}

If we apply \eqref{eq:T-bound} to the polynomial $P(x) = x^m$ for any $m\in\N$, we may deduce that
$$\|T(x^m)\|_2 \leq \|x\|^m \|\delta^\ast(1\otimes 1)\|_2 + m\|x\|^{m-1} \|\delta(x)\|_2$$
since $\|(\partial P)(x)\|_\pi \leq m \|x\|^{m-1}$ holds. From this, we immediately get that
$$\limsup_{m\rightarrow\infty} \|T(x^m)\|_2^{\frac{1}{m}} \leq \|x\|.$$
Thus, condition (ii) of Lemma \ref{lem:iterative-norm-bound} is satisfied.

Lemma \ref{lem:iterative-norm-bound} tells us now that $\|T(x)\|_2 \leq \|\delta^\ast(1\otimes1)\|_2 \|x\|$, which is by definition of $T$ exactly the first inequality in \eqref{eq:Dab-inequalities-1}. The second one can simply be deduced from the first one by using that $\delta^\ast(u^\dagger) = \delta^\ast(u)^\ast$ holds for any $u\in D(\delta^\ast)$ according to Lemma \ref{lem:real-derivations-formulas}, since $\delta$ was assumed to be real.
\end{proof}

Combining Theorem \ref{thm:Dab-lemma} with Lemma \ref{lem:Voi-lemma} yields the following corollary.

\begin{corollary}\label{cor:Dab-extended}
Let $\delta: M \supseteq D(\delta) \to L^2(M,\tau) \otimes L^2(M,\tau)$ a non-commutative derivation on a tracial $W^\ast$-probability space $(M,\tau)$. We assume that $\delta$ is a real derivation in the sense of \ref{def:noncommutative-derivation-real} and that is satisfies the coassociativity relation formulated in \ref{def:noncommutative-derivation-coassociative}. Then, for all $x_1,x_2\in D(\delta)$, it holds true that
\begin{equation}\label{eq:bound-1}
\|\delta^\ast(x_1 \otimes x_2)\|_2 \leq 3 \|\delta^\ast(1\otimes1)\|_2 \|x_1\| \|x_2\|
\end{equation}
and
\begin{equation}\label{eq:bound-2}
\begin{aligned}
\|(\id\otimes\tau)(\delta(x_1) \cdot x_2)\|_2 & \leq 4 \|\delta^\ast(1\otimes1)\|_2 \|x_1\| \|x_2\|,\\
\|(\tau\otimes\id)(x_1 \cdot \delta(x_2))\|_2 & \leq 4 \|\delta^\ast(1\otimes1)\|_2 \|x_1\| \|x_2\|.
\end{aligned}
\end{equation}
\end{corollary}

\begin{proof}
According to Lemma \ref{lem:Voi-lemma}, we have for all $x_1,x_2 \in D(\delta)$ that
\begin{align*}
 \delta^\ast (x_1 \otimes x_2)
 &= x_1 \delta^\ast(1\otimes1) x_2 - m_1(\id\otimes\tau\otimes\id)(\delta \otimes \id + \id \otimes \delta) (x_1 \otimes x_2)\\
 &= x_1 \delta^\ast(1\otimes1) x_2 - (\id\otimes\tau)(\delta(x_1)) x_2 - x_1 (\tau\otimes\id)(\delta(x_2))\\
 &= \delta^\ast(x_1 \otimes 1) x_2 - x_1 (\tau\otimes\id)(\delta(x_2))
\end{align*}
and thus, by applying the estimates \eqref{eq:Dab-inequalities-2} and \eqref{eq:Dab-inequalities-1}, that
\begin{align*}
\|\delta^\ast(x_1 \otimes x_2)\|_2
 &\leq \|\delta^\ast(x_1 \otimes 1)\|_2 \|x_2\| + \|x_1\| \|(\tau\otimes\id)(\delta(x_2))\|_2\\
 &\leq 3 \|\delta^\ast(1\otimes1)\|_2 \|x_1\| \|x_2\|.
\end{align*}
This shows the validity of \eqref{eq:bound-1}. For proving \eqref{eq:bound-2}, we first use integration by parts in order to obtain
\begin{align*}
(\id\otimes\tau)(\delta(x_1) \cdot x_2)
 &= (\id\otimes\tau)(\delta(x_1 x_2)) - (\id\otimes\tau)(x_1 \cdot \delta(x_2))\\
 &= (\id\otimes\tau)(\delta(x_1 x_2)) - x_1 (\id\otimes\tau)(\delta(x_2))
\end{align*}
for arbitrary $x_1,x_2\in D(\delta)$. From this, we can easily deduce by using \eqref{eq:Dab-inequalities-2} that
\begin{align*}
\|(\id\otimes\tau)(& \delta(x_1) \cdot x_2)\|_2\\
 &\leq \|(\id\otimes\tau)(\delta(x_1 x_2))\|_2 + \|x_1\| \|(\id\otimes\tau)(\delta(x_2))\|_2\\
 &\leq 4 \|\delta^\ast(1\otimes1)\|_2 \|x_1\| \|x_2\|
\end{align*}
which is the first inequality of \eqref{eq:bound-2}. The second inequality can either be proven similarly or can be deduced from the first one by using that $\delta$ is real.
\end{proof}

We conclude this subsection by highlighting Formula \eqref{eq:T-formula}, which was obtained in the proof of Theorem \ref{thm:Dab-lemma}. Since we think that this observation might be of independent interest and could be helpful for future investigations, we record \eqref{eq:T-formula} here by the following corollary.

\begin{corollary}\label{cor:T-formula}
Let $\delta: M \supseteq D(\delta) \to L^2(M,\tau) \otimes L^2(M,\tau)$ be a non-commutative derivation on a tracial $W^\ast$-probability space $(M,\tau)$, which is real and satisfies the coassociativity relation. Assume additionally that $1 \otimes 1 \in D(\delta^\ast)$. Then, for each $x\in D(\delta)$, it holds true that
$$\|\delta^\ast(x \otimes 1)\|^2_2 = \langle \delta^\ast( (x^\ast x) \otimes 1), \delta^\ast(1\otimes 1)\rangle.$$
\end{corollary}

Assume, for instance, that in the situation of Corollary \ref{cor:T-formula} the conditions
$$\delta^\ast(1\otimes 1) \in D(\overline{\delta}) \cap M \qquad\text{and}\qquad \overline{\delta}(\delta^\ast(1\otimes 1)) \in M \otimes M$$
are satisfied in addition. Corollary \ref{cor:T-formula} allows us then to conclude that for any $x\in D(\delta)$
$$\|\delta^\ast(x \otimes 1)\|^2_2 = \langle \delta^\ast( (x^\ast x) \otimes 1), \delta^\ast(1\otimes 1)\rangle = \langle (x^\ast x) \otimes 1, \overline{\delta}(\delta^\ast(1\otimes 1))\rangle = \langle x \otimes 1, x \cdot \overline{\delta}(\delta^\ast(1\otimes 1))\rangle$$
and hence $\|\delta^\ast(x \otimes 1)\|_2 \leq \|\overline{\delta}(\delta^\ast(1\otimes 1))\|^{1/2} \|x\|_2$ holds. Like in Theorem \ref{thm:Dab-lemma}, we can use this in combination with \eqref{eq:Voi-formula-red} in order to deduce that
$$\|(\id\otimes\tau)(\delta(x))\|_2 \leq \big(\|\delta^\ast(1\otimes 1)\| + \|\overline{\delta}(\delta^\ast(1\otimes 1))\|^{1/2}\big) \|x\|_2$$
holds for each $x\in D(\delta)$. Analogous inequalities can of course be proven for $\delta^\ast(1\otimes x)$ and $(\tau\otimes\id)(\delta(x))$. In other words, we can strengthen the bounds that were obtained in Theorem \ref{thm:Dab-lemma} by imposing some stronger ``regularity conditions'' on $\delta^\ast(1\otimes 1)$. Note that this in fact slightly improves similar estimates that were deduced in \cite{Dab-free_stochastic_PDE}.

\subsection{Survival of zero divisors}

We are mainly interested here in applications of the theory of non-commutative derivations to regularity questions for certain distributions. The basic idea that originates in \cite{MaiSpeicherWeber2014, MaiSpeicherWeber2015} is that, in order to exclude atoms, one should reformulate this question in more algebraic terms as a question about the existence of zero-divisors, where the latter can be excluded by a successive reduction of the degree by applying non-commutative derivations.

The key for this purpose is a certain inequality which allows the conclusion that zero-divisors $xu=0$ survive under applying operators of the form
$$\Delta_p(x) := (\tau \otimes \id)(p \cdot \delta(x))$$
for any non-commutative derivation $\delta$ satisfying certain conditions and some non-trivial projection $p$. This inequality will be given below in Proposition \ref{prop:key-inequality}. As we will see, it will more generally relate products $xu$ and $x^\ast v$ for elements $x$ in the domain of the given non-commutative derivation $\delta$ and arbitrary elements $u,v$ in the corresponding von Neumann algebra with an expression of the form $v^\ast \cdot \delta(x) \cdot u$.

We point out that although the inequality itself holds in a considerably large generality, the feasibility of the whole strategy for excluding zero-divisors relies heavily on the structure of the given non-commutative derivation. Roughly speaking, applying $\delta$ has to ``reduce the degree'' of the given element $x$. More formally, one should think of a grading on the space of distributions under consideration that is compatible with $\delta$. We do not want to give a definition in full generality, however we want to mention that the grading that was used in \cite{MaiSpeicherWeber2014, MaiSpeicherWeber2015} was given by the monomials of fixed degree. As we will see in Section \ref{sec:ProofMainThm}, where we present the proof of Theorem \ref{MainThm}, there is a closely related grading on the space of finite Wigner integrals. 

The crucial inequality will now be formulated in the following proposition. 

\begin{proposition}\label{prop:key-inequality}
Let $\delta: L^2(M,\tau) \supseteq D(\delta) \to L^2(M,\tau) \otimes L^2(M,\tau)$ be a non-commutative derivation. We assume that $\delta$ is real and satisfies the coassociativity relation.

Then, if in addition $1\otimes 1 \in D(\delta^\ast)$ holds, we have for all $x\in D(\overline{\delta})$, where $\overline{\delta}$ denotes the closure of $\delta$, and $u,v\in M$ the inequality
\begin{equation}\label{eq:key-inequality}
|\langle v^\ast \cdot \overline{\delta}(x) \cdot u, y_1 \otimes y_2\rangle| \leq 4 \|\delta^\ast(1\otimes 1)\|_2 \big(\|v\| \|x u\|_2 + \|u\| \|x^\ast v\|_2\big) \|y_1\| \|y_2\|
\end{equation}
for all $y_1,y_2 \in D(\delta)$.

In particular, if we have both $xu=0$ and $x^\ast v = 0$ for $x\in D(\overline{\delta})$ with some $u,v\in M$, then also $v^\ast \cdot \overline{\delta}(x) \cdot u = 0$ holds.
\end{proposition}

Before giving the proof of Proposition \ref{prop:key-inequality}, we first mention an easy but useful application of Kaplansky's density theorem.

\begin{lemma}\label{lem:Kaplansky}
In the given setting of a tracial $W^\ast$-probability space $(M,\tau)$, let $D$ be a $\ast$-subalgebra of $M$, which is weakly dense in $M$. Then, for each $w\in M$, there exists a sequence $(w_k)_{k\in\N}$ of elements in $D$ such that
\begin{itemize}
\item[(i)] $\displaystyle{\sup_{k\in\N} \|w_k\| \leq \|w\|}$,
\item[(ii)] $\|w_k-w\|_2 \to 0$ as $k \to \infty$.
\end{itemize}
If $w=w^\ast$, then we may assume in addition that $w_k=w_k^\ast$ for all $k\in\N$.
\end{lemma}

\begin{proof}
First of all, we note that for proving the existence of a sequence $(w_k)_{k\in\N}$ of elements in $D$, which satisfies conditions (i) and (ii), it suffices to find a net $(w_\lambda)_{\lambda\in\Lambda}$ of elements in $D$, which satisfies
\begin{itemize}
\item[(i)'] $\displaystyle{\sup_{\lambda\in\Lambda} \|w_\lambda\| \leq \|w\|}$,
\item[(ii)'] $\|w_\lambda-w\|_2 \stackrel{\lambda\in\Lambda}{\longrightarrow} 0$.
\end{itemize}
Indeed, given such a net $(w_\lambda)_{\lambda\in\Lambda}$, we may choose a sequence $(\lambda_k)_{k\in\N}$ in $\Lambda$, such that $\|w_{\lambda_k}-w\|_2 < \frac{1}{k}$ holds for all $k\in\N$. Hence, the sequence $(w_{\lambda_k})_{k\in\N}$ satisfies (i) and (ii), as desired. 

Now, for finding a net of elements in $D$, which satisfies (i)' and (ii)', we apply Kaplansky's density theorem. Indeed, this theorem guarantees the existence of a net $(w_\lambda)_{\lambda\in\Lambda}$ of elements in $D$, such that $\|w_\lambda\| \leq \|w\|$ holds for all $\lambda\in\Lambda$, and which converges to $w$ in the strong operator topology. Thus, the net $(w_\lambda)_{\lambda\in\Lambda}$ already satisfies condition (i)' and it remains to show the validity of (ii)'.

For seeing (ii)', we note that with respect to the weak operator topology,
$$w_\lambda^\ast w \stackrel{\lambda\in\Lambda}{\longrightarrow} w^\ast w,\qquad w^\ast w_\lambda \stackrel{\lambda\in\Lambda}{\longrightarrow} w^\ast w, \qquad\text{and}\qquad w_\lambda^\ast w_\lambda \stackrel{\lambda\in\Lambda}{\longrightarrow} w^\ast w,$$
such that according to the continuity of $\tau$
\begin{align*}
\|w_\lambda-w\|_2^2 &= \tau((w_\lambda-w)^\ast(w_\lambda-w))\\
                    &= \tau(w_\lambda^\ast w_\lambda) - \tau(w_\lambda^\ast w) - \tau(w^\ast w_\lambda) + \tau(w^\ast w)\\
                    &\stackrel{\lambda\in\Lambda}{\longrightarrow} 0,
\end{align*}
as claimed in (ii)'. This concludes the proof of the first part of the lemma.

For proving the additional statement, we just have to observe that in the case $w=w^\ast$, we can take any sequence $(w_k)_{k\in\N}$ that satisfies (i) and (ii), and replace each $w_k$ by its real part $\Re(w_k)=\frac{1}{2}(w_k+w_k^\ast)$. Indeed, for the sequence $(w_k)_{k\in\N}$ obtained in this way, conditions (i) and (ii) are still valid, but we have achieved $w_k=w_k^\ast$ for all $k\in\N$ in addition.
\end{proof}

Now, we may proceed by

\begin{proof}[Proof of Proposition \ref{prop:key-inequality}]
Firstly, we assume that $x\in D(\delta)$ as well as $u,v\in D(\delta)$. In this particular case, we may compute
\begin{align*}
\langle x u, \delta^\ast(v y_1 \otimes y_2)\rangle
&= \langle \delta(x u), v y_1 \otimes y_2\rangle\\
&= \langle \delta(x) \cdot u, v y_1 \otimes y_2\rangle + \langle x \cdot \delta(u), v y_1 \otimes y_2\rangle\\
&= \langle v^\ast \cdot \delta(x) \cdot u, y_1 \otimes y_2\rangle + \langle \delta(u) \cdot y_2^\ast, x^\ast v y_1 \otimes 1\rangle\\
&= \langle v^\ast \cdot \delta(x) \cdot u, y_1 \otimes y_2\rangle + \langle (\id\otimes\tau)(\delta(u)\cdot y_2^\ast), x^\ast v y_1 \rangle.
\end{align*}
Rearranging the terms yields
$$\langle v^\ast \cdot \delta(x) \cdot u, y_1 \otimes y_2\rangle = \langle x u, \delta^\ast(v y_1 \otimes y_2)\rangle - \langle (\id\otimes\tau)(\delta(u)\cdot y_2^\ast), x^\ast v y_1 \rangle,$$
from which we deduce by the inequalities in Corollary \ref{cor:Dab-extended} that
\begin{align*}
|\langle v^\ast \cdot \delta(x) & \cdot u, y_1 \otimes y_2\rangle|\\
&\leq |\langle x u, \delta^\ast(v y_1 \otimes y_2)\rangle| + |\langle (\id\otimes\tau)(\delta(u)\cdot y_2^\ast), x^\ast v y_1 \rangle|\\
&\leq \|xu\|_2 \|\delta^\ast(v y_1 \otimes y_2)\|_2 + \|(\id\otimes\tau)(\delta(u)\cdot y_2^\ast)\|_2 \|x^\ast v y_1\|_2\\
&\leq 4 \|\delta^\ast(1\otimes 1)\|_2 \big(\|v\| \|x u\|_2 + \|u\| \|x^\ast v\|_2\big) \|y_1\| \|y_2\|,
\end{align*}
as desired. Due to Lemma \ref{lem:Kaplansky}, this inequality extends to arbitrary $u,v\in M$.

Thus, we have proven \eqref{eq:key-inequality} for $x\in D(\delta)$ and $u,v\in M$. It remains to show that we may extend it from $x\in D(\delta)$ to $x\in D(\overline{\delta})$.

Since $D(\overline{\delta})$ turns out to be the closure of $D(\delta)$ with respect to the norm $\|\cdot\|_{2,1}$ defined by
$$\|x\|_{2,1} := \big(\|x\|_2^2 + \|\delta(x)\|_2^2\big)^{\frac{1}{2}} \qquad \text{for any $x\in D(\delta)$},$$
we can find for any $x\in D(\overline{\delta})$ a sequence $(x_k)_{k\in\N}$ in $D(\delta)$ such that both conditions $\|x_k - x\|_2 \to 0$ and $\|\delta(x_k) - \overline{\delta}(x)\|_2 \to 0$ as $k\to\infty$ are satisfied. Hence, for given $u,v\in M$, we observe
$$\lim_{k\to\infty} \langle v^\ast \cdot \delta(x_k) \cdot u, y_1 \otimes y_2\rangle = \langle v^\ast \cdot \delta(x_k) \cdot u, y_1 \otimes y_2\rangle$$
and
$$\lim_{k\to\infty} \big(\|v\| \|x_k u\|_2 + \|u\| \|x^\ast_k v\|_2\big) = \|v\| \|x u\|_2 + \|u\| \|x^\ast v\|_2,$$
from which \eqref{eq:key-inequality} immediately follows in full generality.

Finally, if we have $x u = 0$ and $x^\ast v = 0$, then \eqref{eq:key-inequality} implies that
$$\langle v^\ast \cdot \overline{\delta}(x) \cdot u, y_1 \otimes y_2\rangle = 0 \qquad \text{for all $y_1,y_2 \in D(\delta)$}$$
and hence by linearity 
$$\langle v^\ast \cdot \overline{\delta}(x) \cdot u, w\rangle = 0 \qquad \text{for all $w \in D(\delta) \odot D(\delta)$}.$$
Since $D(\delta) \odot D(\delta)$ is dense in $L^2(\S,\tau) \otimes L^2(\S,\tau)$, we obtain $v^\ast \cdot \overline{\delta}(x) \cdot u = 0$, as stated.
\end{proof}

\section{Proof of Theorem \ref{MainThm}}
\label{sec:ProofMainThm}

We are prepared now to build the proof of Theorem \ref{MainThm} on its pillars raised in the previous sections.

In the light of free Malliavin calculus, it seems natural that methods from Section \ref{sec:Derivations} could be used for a proof of Theorem \ref{MainThm} based on the same reduction method as in \cite{MaiSpeicherWeber2014,MaiSpeicherWeber2015}. Nevertheless, there is the fundamental obstacle that in the world of free stochastic calculus, the role of non-commutative derivatives which were used in the ``discrete setting'' of \cite{MaiSpeicherWeber2014,MaiSpeicherWeber2015}, is taken over by the Malliavin operators as their ``continuous counterparts''. These operators are seemingly of completely different nature.

But on closer inspection, it turns out that the right object for this purpose, which bridges -- somehow as an architrave, if one wants to strain the architecture language again -- between free stochastic calculus and the theory of non-commutative derivatives are directional gradients. We will introduce this concept in the following subsection.

\subsection{Directional gradients}

Roughly speaking, directional gradients are obtained from the gradient operator by integrating out the (for us obstructive) time dependence against any function in $L^2(\R_+)$. More formally, we shall introduce these objects as follows.

\begin{definition}\label{def:directional-gradient}
For each $h\in L^2(\R_+)$, we define an unbounded linear operator
$$\nabla^h:\ L^2(\S,\tau) \supseteq D(\nabla^h) \to L^2(\S,\tau) \otimes L^2(\S,\tau)$$
with domain $D(\nabla^h) := D(\nabla) = \S_\alg$ by
$$\nabla^h Y := \langle\nabla Y,h\rangle = \int_{\R_+} \nabla_t Y\, \overline{h(t)}\, dt,$$
where we refer to the pairing $\langle\cdot,\cdot\rangle$ that was introduced in Definition \ref{def:biprocess-pairing}. We call $\nabla^h$ the \emph{directional gradient (in the direction $h$)}.
\end{definition}

This terminology goes in fact parallel to classical Malliavin calculus, where corresponding expressions are also interpreted as directional derivatives.

We collect some basic but very important properties of directional gradients in the following lemma.

\begin{lemma}\label{lem:directional-gradient-basic}
Let $h\in L^2(\R_+)$ be given.
\begin{itemize}
 \item[(a)] If $\cdot$ denotes the left and right action of $\S$ on $L^2(\S,\tau) \otimes L^2(\S,\tau)$, respectively, then the Leibniz rule $$\nabla^h(Y_1 Y_2) = (\nabla^h Y_1) \cdot Y_2 + Y_1 \cdot (\nabla^h Y_2)$$ holds for all $Y_1,Y_2\in D(\nabla^h) = \S_\alg$.
 \item[(b)] For all $Y\in D(\nabla^h)$, it holds true that $$\nabla^h(Y^\ast) = (\nabla^{\overline{h}} Y)^\dagger.$$
Thus, if $h\in L^2(\R_+,\R)$, we have in particular that $$\nabla^h(Y^\ast) = (\nabla^h Y)^\dagger$$ holds for all $Y\in D(\nabla^h)$.
 \item[(c)] The directional gradient $\nabla^h$ takes its values in $\S_\alg \odot \S_\alg$ and we have that $$(\nabla^h \otimes \id) \nabla^h = (\id \otimes \nabla^h) \nabla^h.$$ More generally, it holds true for all $h_1,h_2\in L^2(\R_+)$ that $$(\nabla^{h_1} \otimes \id) \nabla^{h_2} = (\id \otimes \nabla^{h_2}) \nabla^{h_1}.$$
\end{itemize}
\end{lemma}

\begin{proof}
The fact that $\nabla^h$ satisfies the Leibniz rule stated in (a) follows immediately from the Leibniz rule \eqref{eq:Leibniz-rule} for $\nabla$ on $D(\nabla)$, since the domains $D(\nabla)$ and $D(\nabla^h)$ agree.

For seeing (b), we consider $Y=X(h_1)\dots X(h_n) \in \S_\alg$ for $h_1,\dots,h_n\in L^2(\R_+,\R)$. A straightforward calculation confirms that
\begin{align*}
\nabla^h(Y^\ast)
&= \sum^n_{j=1} \langle h_j, h\rangle X(h_n) \cdots X(h_{j+1}) \otimes X(h_{j-1}) \cdots X(h_1)\\
&= \Big(\sum^n_{j=1} \langle h_j, \overline{h}\rangle X(h_1) \cdots X(h_{j-1}) \otimes X(h_{j+1}) \cdots X(h_n)\Big)^\dagger\\
&= (\nabla^{\overline{h}} Y)^\dagger.
\end{align*}
Because $h=\overline{h}$ holds for any $h\in L^2(\R_+,\R)$, the additional statement in (b) is an immediate consequence of the formula  $\nabla^h(Y^\ast) = (\nabla^{\overline{h}} Y)^\dagger$.
Alternatively, by referring to Remark \ref{rem:real-generators}, it suffices to check $\nabla^h(Y^\ast) = (\nabla^{\overline{h}} Y)^\dagger$ on the algebraic generators $(X(g))_{g\in L^2(\R_+,\R)}$ of $\S_\alg$. But in this case, the statement is obvious since $X(g)$ is self-adjoint and since we have $\nabla^h X(g) = \langle g,h\rangle_{L^2(\R_+)}\, 1 \otimes 1$ for any $g\in L^2(\R_+,\R)$.

For proving (c), since $\nabla^h$ clearly takes its values in $\S_\alg \odot \S_\alg$, it only remains to show the stated formula. For doing this, it suffices by linearity to prove
$$(\nabla^{h_1} \otimes \id) \nabla^{h_2} Y = (\id \otimes \nabla^{h_2}) \nabla^{h_1}Y$$
for all $h_1,h_2\in L^2(\R_+)$ and any element $Y\in\S_\alg$ of the form
$$Y = X(g_1) X(g_2) \cdots X(g_n).$$
If $1\leq j_1 < j_2 \leq n$ are given, we will abbreviate in the following
$$\check{X}_{j_1,j_2} := X(g_1) \cdots X(g_{j_1-1}) \otimes X(g_{j_1+1}) \cdots X(g_{j_2-1}) \otimes X(g_{j_2+1}) \cdots X(g_n),$$
where as usually empty products are understood as $1$. Firstly, we compute
$$\nabla^{h_2} Y = \sum_{1\leq j_2 \leq n} \langle g_{j_2},h_2 \rangle X(g_1) \cdots X(g_{j_2-1}) \otimes X(g_{j_2+1}) \cdots X(g_n),$$
which yields
$$(\nabla^{h_1} \otimes \id) \nabla^{h_2} Y = \sum_{1\leq j_1 < j_2 \leq n} \langle g_{j_1}, h_1\rangle \langle g_{j_2},h_2 \rangle \check{X}_{j_1,j_2}$$
Similarly, we compute
$$\nabla^{h_1} Y = \sum_{1 \leq j_1 \leq n} \langle g_{j_1},h_1 \rangle X(g_1) \cdots X(g_{j_1-1}) \otimes X(g_{j_1+1}) \cdots X(g_n),$$
which yields
$$(\id \otimes \nabla^{h_2}) \nabla^{h_1} Y = \sum_{1\leq j_1 < j_2 \leq n} \langle g_{j_1}, h_1\rangle \langle g_{j_2},h_2 \rangle \check{X}_{j_1,j_2}$$
Because the right hand sides of both results agree, we finally obtain the desired equality. This concludes the proof.
\end{proof}

Combining the properties of directional gradients that we have established in the previous Lemma \ref{lem:directional-gradient-basic} leads us immediately to the following crucial observation.

\begin{corollary}\label{cor:directional-gradient-derivation}
For any $h\in L^2(\R_+)$, the directional gradient
$$\nabla^h:\ L^2(\S,\tau) \supseteq D(\nabla^h) \to L^2(\S,\tau) \otimes L^2(\S,\tau),$$
induces a non-commutative derivation on $\S$ in the sense of Definition \ref{def:noncommutative-derivation}, which satisfies additionally the coassociativity relation that was formulated in Definition \ref{def:noncommutative-derivation-coassociative}.
If we choose particularly any $h\in L^2(\R_+,\R)$, then $\nabla^h$ is also a real derivation in the sense of Definition \ref{def:noncommutative-derivation-real}.
\end{corollary}

The importance of this observations is perfectly clear now, since it puts directional gradients in the setting non-commutative derivations and gives therefore access to the general theory that was presented in Section \ref{sec:Derivations}.

However, there is still one key property missing that is needed to fully open this powerful toolbox, namely the condition $1 \otimes 1 \in D(\delta^h)$, where $\delta^h$ denotes the adjoint operator of $\nabla^h$, i.e.
$$\delta^h := (\nabla^h)^\ast:\ L^2(\S,\tau) \otimes L^2(\S,\tau) \supseteq D(\delta^h) \to L^2(\S,\tau),$$
We shall call $\delta^h$ the \emph{directional divergence operator (in the direction $h$)} in the following.

The condition $1 \otimes 1 \in D(\delta^h)$ would in particular guarantee according to Proposition \ref{lem:Voi-lemma} that $\delta^h$ is densely defined and hence that $\nabla^h$ is closable. But there is actually a shortcut in our situation. We insert here the following lemma which expresses the directional divergence operator $\delta^h$ in terms of the divergence operator $\delta$ and which will allow us to conclude directly that the domain of $\delta^h$ is sufficiently large.

\begin{lemma}\label{lem:directional-divergence}
For any $h\in L^2(\R_+)$, the domain $D(\delta^h)$ of the directional divergence operator $\delta^h$ contains $\S_\alg \odot \S_\alg$ and we have explicitly
$$\delta^h(U) = \delta(U \sharp 1 \otimes h \otimes 1) \qquad\text{for all $U\in\S_\alg \odot \S_\alg$}.$$
In particular, $\delta^h$ is densely defined and we have that $1\otimes 1 \in D(\delta^h)$ with $\delta^h(1\otimes 1) = X(h)$.
\end{lemma}

\begin{proof}
We just have to note that by definition $U\sharp 1 \otimes h \otimes 1 \in D(\delta)$ for any $U\in\S_\alg \odot \S_\alg$ and that the corresponding element $\delta(U \sharp 1 \otimes h \otimes 1) \in L^2(\S,\tau)$ satisfies
$$\langle Y, \delta(U \sharp 1 \otimes h \otimes 1) \rangle = \langle \nabla Y, U \sharp 1 \otimes h \otimes 1\rangle_{\B_2} = \langle \nabla^h Y, U\rangle.$$
This means that $\S_\alg \odot \S_\alg \subseteq D(\delta^h)$ and even more explicit
$$\delta^h(U) = \delta(U \sharp 1 \otimes h \otimes 1) \qquad\text{for all $U\in\S_\alg \odot \S_\alg$}.$$
In particular, we may deduce that $\delta^h$ is densely defined and that $1\otimes 1 \in D(\delta^h)$ holds true with $\delta^h(1\otimes 1) = \delta(1 \otimes h \otimes 1) = X(h)$.
\end{proof}

The closability of $\nabla^h$, which is implied by the lemma above, will be recorded in the following proposition. But we discuss there in addition that the domain of the closure of $\nabla^h$ contains the domain of the closure of $\nabla$.

\begin{proposition}\label{prop:directional-gradient}
Given $h\in L^2(\R_+)$. The directional gradient
$$\nabla^h:\ L^2(\S,\tau) \supseteq D(\nabla^h) \to L^2(\S,\tau) \otimes L^2(\S,\tau)$$
is densely defined and closable. The domain $D(\overline{\nabla}^h)$ of its closure
$$\overline{\nabla}^h:\ L^2(\S,\tau) \supseteq D(\overline{\nabla}^h) \to L^2(\S,\tau) \otimes L^2(\S,\tau)$$
contains the domain $D(\overline{\nabla})$ of $\overline{\nabla}$.
\end{proposition}

\begin{proof}
Basic functional analysis tells us that in this case $D(\overline{\nabla}^h)$ is obtained as the closure of $\S_\alg$ with respect to the norm $\|\cdot\|^h_{2,1}$ that is given by 
$$\|Y\|^h_{2,1} := \big(\|Y\|_2^2 + \|\nabla^h Y\|_2^2\big)^{\frac{1}{2}} \qquad\text{for all $Y\in \S_\alg$},$$
whereas the domain $D(\overline{\nabla})$ of $\overline{\nabla}$ is obtained as the closure of $\S_\alg$ with respect to the norm
$$\|Y\|_{2,1} = \big(\|Y\|_2^2 + \|\nabla Y\|_{\B_2}^2\big)^{\frac{1}{2}} \qquad\text{for all $Y\in \S_\alg$},$$
as we pointed out in Remark \ref{rem:domain-gradient}. Therefore, the desired inclusion $D(\overline{\nabla}) \subseteq D(\overline{\nabla}^h)$ follows as soon as we have established that
\begin{equation}\label{domain-inclusion}
\|Y\|^h_{2,1} \leq \max\{1,\|h\|_{L^2(\R_+)}\}\, \|Y\|_{2,1} \qquad\text{for all $Y\in\S_\alg$}.
\end{equation}
For that purpose, we make use of Lemma \ref{lem:biprocess-pairing}. This yields
$$\|\nabla^h Y\|_2 = \|\langle \nabla Y, h\rangle\|_2 \leq \|h\|_{L^2(\R_+)} \|\nabla Y\|_{\B_2}.$$
Now, the desired inequality \eqref{domain-inclusion} immediately follows.
\end{proof}

\subsection{Reduction by directional gradients}

In the previous subsection, we have seen that directional gradients fit nicely into the general frame of non-commutative derivations. The following proposition, which will be a the core of our reduction method, is therefore an immediate consequence of Proposition \ref{prop:key-inequality}.

\begin{proposition}\label{prop:reduction}
Take any $Y\in\S_\fin$. If there are $u,v\in\S$ such that the conditions $Y u = 0$ and $Y^\ast v = 0$ are satisfied, then it holds true that
$$v^\ast \cdot (\overline{\nabla}^h Y) \cdot u = 0 \qquad\text{for all $h\in L^2(\R_+,\R)$}.$$
\end{proposition}

\begin{proof}
Let $h\in L^2(\R_+,\R)$ be given. Firstly, we recall that the directional gradient
$$\nabla:\ L^2(\S,\tau) \supseteq D(\nabla^h) \to L^2(\S,\tau) \otimes L^2(\S,\tau),$$
induces according to Corollary \ref{cor:directional-gradient-derivation} a real non-commutative derivation, which satisfies in addition the coassociativity relation. Furthermore, its adjont operator, the directional divergence operator $\delta^h$, satisfies due to Lemma \ref{lem:directional-divergence} the condition $1 \otimes 1 \in D(\delta^h)$. Thus, we can apply Proposition \ref{prop:key-inequality}, which yields the desired statement. 
\end{proof}

\begin{remark}
In the proof of Proposition \ref{prop:reduction} above, we used crucially the properties of directional gradients, which put them nicely in the setting of non-commutative derivations and which therefore allowed us to turn on by Proposition \ref{prop:key-inequality} the powerful machinery that was build up in Section \ref{sec:Derivations}.

But recall that one of the crucial ingredients in the proof of Proposition \ref{prop:key-inequality} were Dabrowski's inequalities \ref{thm:Dab-lemma}. Thus, concealed in the larger apparatus, we deduced particularly for any $Y\in D(\nabla^h)$ according to the inequalities \eqref{eq:Dab-inequalities-1} that
\begin{equation}
\begin{aligned}
\|\delta^h(Y\otimes 1)\|_2 & \leq \|h\|_{L^2(\R_+)} \|Y\|,\\
\|\delta^h(1\otimes Y)\|_2 & \leq \|h\|_{L^2(\R_+)} \|Y\|,
\end{aligned}
\end{equation}
and according to the inequalities \eqref{eq:Dab-inequalities-2} that
\begin{equation}
\begin{aligned}
\|(\id\otimes\tau)(\nabla^h Y)\|_2 &\leq 2 \|h\|_{L^2(\R_+)} \|Y\|,\\
\|(\tau\otimes\id)(\nabla^h Y)\|_2 &\leq 2 \|h\|_{L^2(\R_+)} \|Y\|,
\end{aligned}
\end{equation}
since we have $\|\delta^h(1\otimes1)\|_2 = \|h\|_{L^2(\R_+)}$.

However, the semicircular generators that underlie our situation force in fact a much stronger result than the inequalities above. In fact, for any $Y\in\S_\fin$, we have that
\begin{equation}
\begin{aligned}
\|\delta^h(Y\otimes 1)\|_2 &= \|h\|_{L^2(\R_+)} \|Y\|_2,\\
\|\delta^h(1\otimes Y)\|_2 &= \|h\|_{L^2(\R_+)} \|Y\|_2,
\end{aligned}
\end{equation}
and
\begin{equation}
\begin{aligned}
\|(\id\otimes\tau)(\nabla^h Y)\|_2 &\leq \|h\|_{L^2(\R_+)} \|Y\|_2,\\
\|(\tau\otimes\id)(\nabla^h Y)\|_2 &\leq \|h\|_{L^2(\R_+)} \|Y\|_2.
\end{aligned}
\end{equation}
This can be seen by considering the chaos decomposition of $Y$ and by using the formulas
\begin{equation}
\begin{aligned}
\delta^h(I_n^S(f) \otimes 1) &= I_{n+1}^S(f \otimes h),\\
\delta^h(1 \otimes I^S_n(f)) &= I_{n+1}^S(h \otimes f)
\end{aligned}
\end{equation}
and
\begin{equation}
\begin{aligned}
(\id\otimes\tau)(\nabla^h I^S_n(f)) &= I^S_{n-1}(f \stackrel{1}{\smallfrown} h),\\
(\tau\otimes\id)(\nabla^h I^S_n(f)) &= I^S_{n-1}(h \stackrel{1}{\smallfrown} f).
\end{aligned}
\end{equation}

The author is grateful to Yoann Dabrowski for pointing out that this fact should be included for reasons of clarity.

Of course, one could argue now that in view of this observation, the discussion around Theorem \ref{thm:Dab-lemma} becomes superfluous in the context of this paper. But since there is absolutely no chance to avoid completely a detour through the realm of non-commutative derivations -- even by taking this shortcut -- we decided to present the theory of non-commutative derivations (and in particular the result of Proposition \ref{prop:key-inequality}) in full generality, in order to show the complete picture and to make it ready for its possible use in future investigations.
\end{remark}

\subsection{How to control the reduction}

Because Theorem \ref{MainThm} is a statement about elements $f\in\F$, which break off after finitely many non-zero terms, namely about elements in $\F_\fin := \F_\fin(L^2(\R))$, we shall take now a closer look on
$$\S_\fin := \Big\{I^S(f)\Big|\ f=(f_n)_{n=0}^\infty \in\F_\fin\Big\}$$
as the corresponding space of Wigner integrals, called the \emph{finite Wigner chaos}. By definition, $\S_\fin$ is only a subset of $L^2(\S,\tau)$, but due to \eqref{Haagerup-inequality} and Proposition \ref{thm:Wigner-products} it turns out to be in fact a $\ast$-subalgebra of $\S$. Combining this with the easy observation that $\S_\alg$ is contained in $\S_\fin$, we may localize $\S_\fin$ as intermediate $\ast$-algebra $\S_\alg \subseteq \S_\fin \subseteq \S$.

Following the lines of the proof in \cite{MaiSpeicherWeber2014,MaiSpeicherWeber2015}, we shall introduce now certain operators, which will later allow us to reduce any zero-divisor in $\S_\fin$ in a controllable way to a zero-divisor in the chaos of order zero by means of Proposition \ref{prop:reduction}.

\begin{definition} 
For any $h\in L^2(\R_+)$ and any projection $p\in\S$, we consider the linear operator $\Delta_{p,h}: \S_\fin \to \S_\fin$ that is defined by
$$\Delta_{p,h} Y := (\tau\otimes\id)\big(p \otimes 1\, (\overline{\nabla}^h Y)\big) \qquad \text{for all $Y\in \S_\fin$}.$$
\end{definition}

Note that these operators are indeed well-defined since $\S_\fin \subseteq D(\overline{\nabla}) \subseteq D(\overline{\nabla}^h)$ holds by Proposition \ref{prop:directional-gradient}. The fact that $\Delta_{p,h}$ takes its values in $\S_\fin$ and is made more precise in the following lemma, which moreover shows that $\Delta_{p,h}$ ``reduces the degree'' with respect to the natural grading on $\S_\fin$, which is induced by $\F_\fin$. 

\begin{lemma}\label{lem:Delta-formula}
Let $h\in L^2(\R_+)$ and any projection $p\in\S$ be given. Let $\tau_p$ be the bounded linear functional on $\F$ that is given by
$$\tau_p:\ \F\to\C,\ f\mapsto \tau(p I^S(f)).$$
In fact, if we make use of the chaos decomposition of $p$, we can write $p=I^S(g)$ for some $g=(g_n)_{n=0}^\infty \in \F$, so that $\tau_p(f) = \langle f,g\rangle_\F$ holds for all $f\in\F$.

Now, let $f \in L^2(\R_+^n)$ be given. For $1\leq k\leq n$, we may regard $f$ as an element $f^{(k-1,n-k)}$ in $L^2(\R_+, L^2(\R_+^{k-1}) \otimes L^2(\R_+^{n-k})) \subset L^2(\R_+, \F \otimes \F)$. Using this notation, it holds true that
\begin{equation}\label{eq:Delta-formula}
\Delta_{p,h} I^S_n(f) = \sum^n_{k=1} I^S_{n-k}\Big((\tau_p \otimes \id_\F)\Big(\int_{\R_+} f^{(k-1,n-k)}_t\, \overline{h(t)}\, dt\Big)\Big).
\end{equation}
\end{lemma}

\begin{proof}
It is very easy to check the validity of the formula under question in the case $f = f_1 \otimes \dots \otimes f_n$. Indeed, we have
$$\overline{\nabla}^h I_n^S(f) = \sum^n_{k=1} \langle f_k, h\rangle I^S_{k-1}(f_1 \otimes \dots \otimes f_{k-1}) \otimes I^S_{n-k}(f_{k+1} \otimes \dots \otimes f_n)$$
and hence
\begin{align*}
\Delta_{p,h} I^S_n(f)
&= \sum^n_{k=1} \langle f_k, h\rangle \tau(p I^S_{k-1}(f_1 \otimes \dots \otimes f_{k-1})) I^S_{n-k}(f_{k+1} \otimes \dots \otimes f_n)\\
&= \sum^n_{k=1} I^S_{n-k}\Big(\langle f_k, h\rangle \tau_p(f_1 \otimes \dots \otimes f_{k-1})\, f_{k+1} \otimes \dots \otimes f_n\Big)\\
&= \sum^n_{k=1} I^S_{n-k}\Big((\tau_p \otimes \id_\F)\Big(\int_{\R_+} f^{(k-1,n-k)}_t\, \overline{h(t)}\, dt\Big)\Big),
\end{align*}
which confirms the desired formula \eqref{eq:Delta-formula} in the case $f = f_1 \otimes \dots \otimes f_n$. By linearity of both of its sides, we conclude that formula \eqref{eq:Delta-formula} also holds for any function in the linear span of
$$\{f_1 \otimes \dots \otimes f_n|\ f_1,\dots,f_n\in L^2(\R_+)\},$$
i.e. for any function in $L^2(\R_+)^{\odot n}$. Since this linear space is dense in $L^2(\R_+^n)$ with respect to $\|\cdot\|_{L^2(\R_+^n)}$, it remains to note that \eqref{eq:Delta-formula} stays valid under taking limits with respect to $\|\cdot\|_{L^2(\R_+^n)}$, which means that we prove the continuity of the left and the right hand side of the formula under question with respect to $\|\cdot\|_{L^2(\R_+^n)}$.

Concerning first the left hand side, we note that
$$\|\overline{\nabla} I^S_n(f)\|_2 = \sqrt{n}\, \|f\|_{L^2(\R^n_+)}.$$
Indeed, we have according to Proposition \ref{prop:number-operator} and the It\^o isometry that
$$\|\overline{\nabla} I^S_n(f)\|_2^2 = \langle \overline{\nabla} I^S_n(f), \overline{\nabla} I^S_n(f)\rangle = \langle \big(\overline{\delta}\, \overline{\nabla}\big) I^S_n(f), I^S_n(f)\rangle = n\, \|I^S_n(f)\|_2^2 = n\, \|f\|_{L^2(\R^n_+)}^2.$$
Thus, we obtain the desired bound
$$\|\Delta_{p,h} I^S_n(f)\|_2 \leq \|p\|\, \|\overline{\nabla} I^S_n(f)\|_2 = \sqrt{n}\, \|p\|\, \|f\|_{L^2(\R^n_+)}.$$

Concerning now the right hand side of the formula under question, we note that 
$$\Big\|\int_{\R_+} f^{(k-1,n-k)}_t\, \overline{h(t)}\, dt\Big\|_{L^2(\R_+^{k-1}) \otimes L^2(\R_+^{n-k})} \leq \|h\|_{L^2(\R_+)} \|f\|_{L^2(\R_+^n)}$$
holds for $1\leq k \leq n$, which yields by the It\^o isometry
\begin{align*}
\Big\|I^S_{n-k}\Big((\tau_p \otimes & \id_\F)\Big(\int_{\R_+} f^{(k-1,n-k)}_t\, \overline{h(t)}\, dt\Big)\Big)\Big\|_2\\
& \leq \Big\|(\tau_p \otimes \id_\F)\Big(\int_{\R_+} f^{(k-1,n-k)}_t\, \overline{h(t)}\, dt\Big)\Big\|_{L^2(\R_+^{n-k})}\\
& \leq \|p\|_2 \Big\|\int_{\R_+} f^{(k-1,n-k)}_t\, \overline{h(t)}\, dt\Big\|_{L^2(\R_+^{k-1}) \otimes L^2(\R_+^{n-k})}\\
& \leq \|p\|_2 \|h\|_{L^2(\R_+)} \|f\|_{L^2(\R_+^n)}.
\end{align*}
This concludes the proof.
\end{proof}

By applying iteratively operators of the form $\Delta_{p,h}$ to a fixed element in the finite Wigner chaos $\S_\fin$, we will therefore reach the chaos of order zero after finitely many steps. The following proposition provides an explicit formula for the output of this procedure.

\begin{proposition}\label{prop:coefficient}
Let $f=(f_n)_{n=0}^\infty \in \F_\fin$ be given and let $N\in\N$ be chosen such that $f_n=0$ for all $n \geq N+1$. Then, for any choice of functions $h_1,\dots,h_N\in L^2(\R_+)$ and projections $p_1,\dots,p_N$, it holds true that
$$\Delta_{p_N,h_N} \cdots \Delta_{p_1,h_1} I^S(f) = \tau(p_1) \cdots\tau(p_N)\, \langle f_N, h_1 \otimes \dots \otimes h_N\rangle\, 1.$$
\end{proposition}

Before continuing with the proof of the general statement, we first focus on the special case of simple functions.

\begin{remark}
We note that for any $Y\in\S_\alg$
\begin{align*}
\Delta_{p_N,h_N} \cdots \Delta_{p_1,h_1} Y = (\tau^{\otimes N} \otimes \id) \big(p_1 \otimes \cdots \otimes p_N \otimes 1\, (\overline{\nabla}^{h_1,\dots,h_N} Y)\big)
\end{align*}
holds, where the \emph{iterated gradient} $\overline{\nabla}^{h_1,\dots,h_N}: \S_\alg \to \S_\alg^{\odot (N+1)}$ is defined by
$$\overline{\nabla}^{h_1,\dots,h_N} := (\id^{\otimes (N-1)} \otimes \overline{\nabla}^{h_N}) \dots (\id \otimes \overline{\nabla}^{h_2}) \overline{\nabla}^{h_1}.$$
Thus, the statement of Proposition \ref{prop:coefficient} becomes apparent in the case where $f=(f_n)_{n=0}^\infty\in\F_\fin$ consists of simple functions $f_n\in\E(\R_+^n)$. Indeed, since $I^S(f)$ decomposes by the conditions that are imposed on $f$ as
$$I^S(f) = I_0^S(f_0) + I_1^S(f_1) + \dots + I_N^S(f_N)$$
and because obviously
$$\overline{\nabla}^{h_1,\dots,h_N} I_k^S(f_k) = 0 \qquad\text{for $0\leq k \leq N-1$},$$
we see that
$$\overline{\nabla}^{h_1,\dots,h_N} I^S(f) = \overline{\nabla}^{h_1,\dots,h_N} I^S_N(f_N).$$
By using Proposition \ref{prop:gradient}, we get
$$\overline{\nabla}^{h_1,\dots,h_N} I_N^S(f_N) = \langle f, h_1 \otimes \dots \otimes h_N\rangle\, 1^{\otimes (N+1)}.$$
Combining these observations yields
$$\Delta_{p_N,h_N} \cdots \Delta_{p_1,h_1} I^S(f) = \tau(p_1) \cdots\tau(p_N)\, \langle f_N, h_1 \otimes \dots \otimes h_N\rangle\, 1,$$
which is the stated formula.
\end{remark}

\begin{proof}[Proof of Proposition \ref{prop:coefficient}]
For the general case, we make use of Lemma \ref{lem:Delta-formula}. Applying formula \eqref{eq:Delta-formula} iteratively, yields that for $1\leq m\leq N$
$$\Delta_{p_m,h_m} \dots \Delta_{p_1,h_1} I^S(f) = I^S(f^{(m)}),$$
for some $f^{(m)}\in\F_\fin$, where $f^{(m)}_n = 0$ for all $n \geq N-m+1$. Moreover, if we put $f^{(0)} := f$, we see that
$$f^{(m)}_{N-m}(t_{m+1},\dots,t_N) = \tau(p_m) \int_{\R_+} f^{(m-1)}_{N-m+1}(t_m,t_{m+1},\dots,t_N)\, \overline{h_m(t_m)}\, dt_m$$
for all $1\leq m \leq N-1$ and
$$f^{(N)}_0 = \tau(p_N) \int_{\R_+} f^{(N-1)}_1(t_N)\, \overline{h_N(t_N)}\, dt_N.$$
Hence, the only term that survives in $\Delta_{p_N,h_N} \dots \Delta_{p_1,h_1} I^S(f)$ is induced by
$$f^{(N)}_0 = \tau(p_1) \cdots\tau(p_N)\, \langle f_N, h_1 \otimes \dots \otimes h_N\rangle,$$
which gives the stated formula.
\end{proof}

\subsection{Absence of zero divisors}

Our discussion in the previous subsections has shown that directional gradients allow us to transfer tools from the theory of non-commutative derivations as presented in Section \ref{sec:Derivations} to the setting of free stochastic calculus. Moreover, we have convinced ourselves that directional gradients $\nabla^h$ induce operators $\Delta_{p,h}$, which satisfy the general conditions for performing our reduction method.

Putting things together, we obtain the following theorem, of which the desired Theorem \ref{MainThm} will be a corollary.

\begin{theorem}\label{thm:no-zero-divisors}
There are no zero divisors in $\S_\fin$. More precisely, if $0 \neq Y\in\S_\fin$ is given, then there is no $0\neq u\in\S$ such that $Yu=0$.
\end{theorem}

\begin{proof}
Contrarily, assume that there are $0 \neq Y\in\S_\fin$ and $0\neq u\in\S$ such that $Yu=0$. We may write $Y=I^S(f)$ for some $f\in\F_\fin$ of the form $f=(f_n)_{n=0}^\infty$. Moreover, we may choose $N\in\N$ such that $f_N \neq 0$ but $f_n = 0$ for all $n\geq N+1$.

Now, we fix arbitrary functions $h_1,\dots,h_N\in L^2(\R_+,\R)$. Recall that whenever we have an element $X\in\S$ such that $Xu=0$ holds, then there exists (since we assumed that $u\neq0$) a non-zero projection $p\in\S$ such that $X^\ast p=0$. This is in fact an easy consequence of the Murray-von Neumann equivalence of the left and right support projections of $X$; see also \cite[Lemma 3.14]{MaiSpeicherWeber2015}. Thus, by applying Proposition \ref{prop:reduction} iteratively, we may find non-zero projections $p_1,\dots,p_N\in\S$ such that
$$(\Delta_{p_N,h_N} \dots \Delta_{p_1,h_1} Y) u = 0.$$
According to Proposition \ref{prop:coefficient}, this means that
$$\tau(p_1) \cdots \tau(p_N)\, \langle f_N, h_1 \otimes \cdots \otimes h_N\rangle\, u = 0.$$
Since we have by assumption $u\neq 0$ and furthermore $\tau(p_1) \cdots \tau(p_N) \neq 0$, because $p_1,\dots,p_N$ are non-zero projections, it follows
$$\langle f_N, h_1 \otimes \cdots \otimes h_N\rangle = 0.$$

Inasmuch as the linear span of
$$\{h_1 \otimes \cdots \otimes h_N|\ h_1,\dots,h_N \in L^2(\R_+,\R)\}$$
is dense in $L^2(\R_+^N)$ with respect to $\|\cdot\|_{L^2(\R_+^N)}$, the previous insight yields $f_N=0$, which contradicts the condition according to which $N$ was chosen. Thus, the assumption made above was wrong, so that the statement of the theorem must be true.
\end{proof}

We finish by showing that Theorem \ref{MainThm} is indeed a consequence of Theorem \ref{thm:no-zero-divisors} above. In fact, we will deduce Theorem \ref{MainThm} exactly in the same way as it was done for the analogous statement in \cite{MaiSpeicherWeber2015}.

\begin{proof}[Proof of Theorem \ref{MainThm}]
More generally, by allowing right from the beginning a constant summand $I^S_0(f_0)$, we show the following: the distribution $\mu_Y$ of any self-adjoint element $Y\in \S_\fin$, which does not belong to the chaos of order zero, cannot have atoms.

Let $Y\in\S_\fin$ be given. If $Y$ does not belong to the chaos of order zero, we can write it as
$$Y = I^S(f) = I^S_0(f_0) + I^S_1(f_1) + \dots + I^S_N(f_N)$$
for some $f=(f_n)_{n=0}^\infty \in \F_\fin$, which is stationary zero after $f_N \neq 0$ for some $N\in\N$. (Note that $N\neq0$ means abstractly speaking that $Y$ is not constant, as it was assumed in \cite{MaiSpeicherWeber2015}.) Then, we observe that any atom $\alpha$ of the distribution $\mu_Y$ of $Y$, i.e. any $\alpha\in\R$ satisfying $\mu_Y(\{\alpha\})\neq 0$, leads by the spectral theorem for bounded self-adjoint operators on Hilbert spaces to a non-zero projection $u$ satisfying $(Y - \alpha 1)u = 0$. Now, Theorem \ref{thm:no-zero-divisors} tells us that $Y = \alpha 1$, which contradicts $f_N\neq 0$.
\end{proof}

\bibliographystyle{amsalpha}
\bibliography{Wigner_integrals}

\end{document}